\documentclass[11pt,dvipsnames]{amsart}

\usepackage[english]{babel}
\usepackage[utf8]{inputenc}
\usepackage{amssymb,amsthm,amsfonts,amsmath}
\usepackage{mathrsfs}
\usepackage{verbatim}
\usepackage{mathtools}
\usepackage[shortlabels]{enumitem}
\usepackage[margin=1in]{geometry}

\usepackage{tikz}
\usepackage[backref=page,colorlinks,breaklinks,linkcolor=BrickRed,citecolor=OliveGreen,urlcolor=blue,hypertexnames=true]{hyperref}
\newcommand{\Paff}{\mathcal{P}_{\text{aff}}(K)}

\newcommand{\Caff}{\mathcal{C}_{\text{aff}}(K)}

\newtheorem{theorem}{Theorem}[section]

\newtheorem{lemma}[theorem]{Lemma}
\newtheorem{proposition}[theorem]{Proposition}

\theoremstyle{definition}
\newtheorem{definition}[theorem]{Definition}
\newtheorem{remark}[theorem]{Remark}
\newtheorem{example}[theorem]{Example}

\DeclareMathOperator{\Aff}{Aff} % The space of affine functions
\newcommand{\R}{\mathbb{R}}
\newcommand{\C}{\mathbb{C}}

\def\cC{ {\mathcal C} }

\def\cM{{\mathcal M}}

\def\cP{{\mathcal P}}

\def\cR{{ \mathcal R }}

\begin{document}
	
	\title{Kadison duality for partially convex sets}

	\author[T. \v{S}trekelj]{Tea \v{S}trekelj}
	\address{Department of Mathematics, FAMNIT, University of Primorska, Koper \&  Institute of Mathematics, Physics and Mechanics, Ljubljana, Slovenia.}
	\email{tea.strekelj@famnit.upr.si}

	\thanks{The author was supported by the Slovenian Research and Innovation Agency grant J1-60011.}

	

	\begin{abstract}
		This paper extends the Kadison duality between compact convex sets and function systems to the setting of partial convexity. A partially convex set is a set that is convex in a designated set of ``convex" variables when the others are held fixed. We introduce the notion of a regular partially convex set and identify its dual as a finitely generated free module over a commutative $C^*$-algebra endowed with a compatible Archimedean order unit structure. We call such spaces free order unit modules. We prove that for any compact regular partially convex set $K$, the space of continuous functions that are affine in the convex variables is the canonical example of such a module. Conversely, we show that (the coordinate realization of) the partially convex state space of a free order unit module is a compact regular partially convex set. Our main result establishes a duality between compact regular partially convex sets and free order unit modules. We also establish a Stone-Weierstrass-type theorem, demonstrating that partially affine polynomials are dense in the space of continuous partially affine functions on any compact regular partially convex set. Finally, we prove a Hahn-Banach-type separation theorem of compact partially convex sets from points outside them.\looseness=-1
	\end{abstract}

	\maketitle
	
	\newpage

\section{Introduction}

Convexity is a simple but powerful concept originating from
geometry. It gives theoretical background to many research areas in mathematics, such as  linear programming, computational geometry and probability.
A subset of a vector space is \textit{convex} if it contains the line segment joining any two of its points, or equivalently, if it is closed under convex combinations.  In modern analysis, the importance of convexity often lies in its dual relationship with the space of affine functions defined upon the set.

The cornerstone result in this direction is the Kadison duality theorem \cite{Al}. Specifically, this duality establishes a canonical bijection between compact convex sets and function systems, which are unital, closed, self-adjoint subspaces of continuous functions. For a compact convex set $K,$ the associated function system is the space of continuous affine functions on $K.$  Conversely, any abstract function system $\mathcal{R}$ finds a geometric realization in its state space $S(\mathcal{R})$, the set of all unital positive functionals on $\mathcal{R}$. 
This correspondence allows for a rich interplay: the properties of the affine functions on $K$ provide insight into the structure of $K$ itself, and in turn, the geometric properties of $K$ are reflected in the behavior of the corresponding affine functions.

Building on this classical foundation, there have been many generalizations of convexity developed for more complex geometric contexts \cite{MM96, Nai04}. 
A well-established noncommutative notion of convexity is matrix convexity \cite{Wit, EW, WW, HM12, HM14, HKM17, HL, DK}. In this setting, single convex sets are replaced by families of matrix sets across all dimensions, connected via direct sums and isometric conjugations.
Just as classical convexity is linked to function systems via Kadison duality, matrix convexity finds its natural noncommutative counterpart in operator systems via Webster-Winkler duality \cite{WW}. Hence, by studying matrix convexity one introduces several convex-geometric ideas and tools to understand operator spaces, systems and algebras \cite{Pa, DM05, Arv08, DK15}.

%For the purposes of this , the most significant of these generalizations is $\Gamma$-convexity, a noncommutative concept that provides the broader framework for the partially convex sets central to this paper. 
Recently, generalized notions of noncommutative convexity have been studied \cite{HHLM, DHM, JKMMP21,JKMMP22, KMS+}, where the underlying sets are convex only in a subset of variables. Formally, the geometry of these sets is encoded by a tuple $\Gamma$ of noncommutative polynomials. The core idea is that while classical convex sets are intersections of half-planes, these $\Gamma$-convex sets are defined by systems of polynomial inequalities involving the components of $\Gamma$. \looseness=-1

The aim of this paper is to establish an analog of Kadison duality for partially convex sets, a prominent instance of $\Gamma$-convex sets. 

\subsection{Partial convexity}
Let $V = \R^{n+m}$ be a finite-dimensional real vector space. An element of $V$ is denoted by $z = (x, y)$ with $x \in \R^n$ and $y \in \R^m$. For a subset $K \subseteq V$ let $\pi_2(K) = \{y \in \R^m \ | \ \exists x\in \R^n: (x,y) \in K\}$ be the projection of $K$ onto $\R^m.$ The projection $\pi_1(K)$ on $\R^n$ is defined similarly.

\begin{definition} A set $K \subseteq \R^{n+m}$ is \textbf{partially convex in $x$} if for any fixed $y_0 \in \pi_2(K),$ the slice $K_{y_0} = \{x \in \R^n \mid (x, y_0) \in K\}$ is a  convex set.
\end{definition}

\begin{example}
	All convex sets are partially convex, independent of the decomposition of the ambient space $\mathbb{R}^d$ (including the trivial decomposition $d = d + 0$). Similarly, a disjoint union of convex sets is partially convex in $x$ as long as the sets do not overlap vertically (i.e., they share no $y$-coordinates). Another example is shown by Figure \ref{fig1} below:~the set bounded by the lines $y = \pm 1$ and the hyperbola $2x^2=y^2+1.$
\end{example}

We will consider spaces of functions on a partially convex set $K$ endowed with the supremum norm $\|f\|_\infty = \sup_{z \in K} |f(z)|$. 

\newpage

\begin{definition}\label{def:morph} Let $K \subseteq \R^{n+m}$ be a partially convex set.
	\begin{enumerate}
		\item The space of \textbf{partially affine polynomials}, $\Paff$, consists of all polynomials $p(x, y)$ such that for any fixed $y_0$, the function $x \mapsto p(x, y_0)$ is affine. Such a polynomial can be written in the form
		\[ p(x, y) = \sum_{i=1}^{n} c_i(y) x_i + c_0(y) \]
		where the $c_i(y)$ are polynomials in the $y$ variables.
		
		\item The space of \textbf{continuous partially affine functions}, $\Caff$, consists of all continuous functions $f(x, y)$ on $K$ such that for any fixed $y_0$, the function $x \mapsto f(x, y_0)$ is affine. 
	\end{enumerate}
\end{definition}

\begin{definition}\label{def:part-iso}
Two partially convex sets $K, L \subseteq \mathbb{R}^{n+m}$ in $x \in \mathbb{R}^{n}$ are \textbf{isomorphic} if there exists a homeomorphism $\Theta: K \rightarrow L$ of the form
\[
\Theta(x, y) = \left( \theta_y(x), \phi(y) \right)
\]
such that
\begin{enumerate}[(i)]
	\item $\phi: \pi_2(K) \rightarrow \pi_2(L)$ is a homeomorphism;
	\item for every fixed $y \in \pi_2(K)$, the map $\theta_y: K_y \rightarrow L_{\phi(y)}$ is an affine homeomorphism.
\end{enumerate}
\end{definition}

\begin{center}
	\begin{figure}[h]
	\begin{tikzpicture}[scale=1.5] 
		% Axes
		\draw[->] (-1.5, 0) -- (1.5, 0) node[right] {$x$};
		\draw[->] (0, -1.5) -- (0, 1.5) node[above] {$y$};
		
		% Plot the curves
		\draw[scale=1, domain=-1:1, smooth, variable=\x, thick, blue!80!black] plot ({\x}, {1});
		\draw[scale=1, domain=-1:1, smooth, variable=\x, thick, blue!80!black] plot ({\x}, {-1});
		\draw[scale=1, domain=-1:1, smooth, variable=\x, thick, blue!80!black] plot ({sqrt((\x*\x+1)/2)},{\x});
		\draw[scale=1, domain=-1:1, smooth, variable=\x, thick, blue!80!black] plot (-{sqrt((\x*\x+1)/2)},{\x});
		
		% Shaded region
		\fill[blue!30, opacity=0.5] plot[domain=-1:1, variable=\x] ({\x}, {1})
		--	plot[domain=-1:1, variable=\x] ({\x}, {-1}) 
		-- plot[domain=-1:1, variable=\y] ({sqrt((\y*\y+1)/2)}, {\y}) 
		-- plot[domain=-1:1, variable=\y] ({-sqrt((\y*\y+1)/2)}, {\y}) --  cycle;
	\end{tikzpicture}\caption{A partially convex set, bounded by the lines $y = \pm 1$ and the hyperbola $2x^2=y^2+1.$}\label{fig1}
	\end{figure}
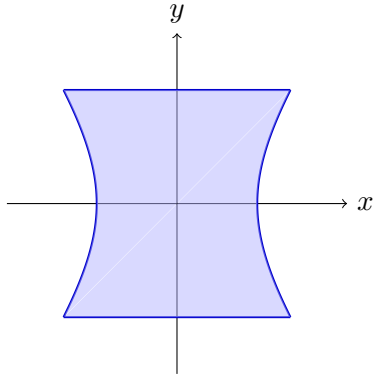
\end{center}

\subsection{Main results}
The main contribution of this paper is the extension of Kadison’s duality between convex sets and function systems to the setting of partial convexity. Precisely, we focus on \textit{regular} partially convex sets as introduced in Definition \ref{def:reg}, where regularity ensures the well-behaved nature of the set’s boundary.

In Theorem \ref{thm:ctsXY} we show that every continuous partially affine function $f$ on a regular partially convex set $K$ (in the $x$ variable) admits a unique representation
\[
f(x,y) = c_0(y) + \sum_{i=1}^n c_i(y)x_i,
\]
where the functions $c_0, c_1, \dots, c_n: Y \to \R$ are continuous (here $Y = \pi_2(K)).$ This implies that continuous partially affine functions on a compact regular partially convex set can be uniformly approximated by partially affine polynomials (Proposition \ref{prop:p-approx}).

Next, Definition \ref{def:module} provides an axiomatic characterization of the space $\Caff$ for a compact regular partially convex set $K$ as a \textit{free order unit module}. In essence, a free order unit module $A$ is both an Archimedean order unit space and a finitely generated free module of rank $n+1$ over the commutative $C^*$-algebra $\mathcal{C}(Y)$. Its structure is determined by local fibers $A_y,$ finite-dimensional spaces of affine functions for each $y \in Y,$ which form a bundle with a hemicontinuous positive cone. To ensure global consistency, we require that the algebraic isomorphism between the module and its space of coefficient functions is a homeomorphism. This module structure formally captures the fact that multiplying a partially affine function by a function of the parameter $y$ preserves partial affinity.

Having established the canonical function system associated with a partially convex set, we conversely introduce canonical partially convex sets arising from free order unit modules (Definition \ref{def:pstate}). The \textit{partially convex state space} $S_{par}(A)$ of a free order unit module $A$ is defined as the union of the state spaces of the individual fibers $A_y$ over the base space $Y$. When equipped with coordinates derived from a module basis of $A,$ this state space yields a compact regular partially convex set $K_A \subseteq \mathbb{R}^{n+m},$ which we call the \textit{partially convex coordinate state space}.

We show that $\Caff$ satisfies the axioms of a free order unit module (Proposition \ref{prop:axiomok}) and, conversely, that the partially convex coordinate state space of such a module is always a compact regular partially convex set (Proposition \ref{prop:Sreg}). These two results are consolidated into our main result, the following duality theorem.

\begin{theorem} \phantom{ } \label{thm:dual}
	\begin{enumerate}

			\item %Let $K \subseteq \mathbb{R}^{n+m}$ be a compact regular partially convex set in $x \in \mathbb{R}^n$. Then $C_{\mathrm{aff}}(K)$ is a free order unit module of type $(n, m)$, and $K$ is isomorphic to the partially convex state space $S_{\mathrm{par}}(C_{\mathrm{aff}}(K))$.
			Let $K\subseteq\mathbb{R}^{n+m}$ be a compact regular partially convex set in $x\in\mathbb{R}^{n}$. Then $\mathcal{C}_{\textnormal{aff}}(K)$ is a free order unit module and $K$ is isomorphic to the partially convex coordinate state space $K_{\mathcal{C}_{\textnormal{aff}}(K)}.$ 
			%(and canonically homeomorphic to $S_{par}(\mathcal{C}_{\textnormal{aff}}(K))$ via fiberwise affine evaluation maps).  
		\item Let $A$ be a free order unit module. Then its partially convex coordinate state space $K_A \subseteq \mathbb{R}^{n+m}$ is a compact regular partially convex set, and $A$ is isomorphic to $C_{\mathrm{aff}}(K_A)$ as a free order unit module.
	\end{enumerate}  
\end{theorem}

The proof of Theorem \ref{thm:dual} is divided into two parts: item \textit{(1)} follows from Theorem \ref{thm:sep} and item \textit{(2)} from Theorem \ref{th:fnsys}.

\subsection*{Reader's guide} The remainder of the paper is organized as follows. 
Section \ref{sec:gamma} reviews the broader theory of $\Gamma$-convexity that motivates partial convexity; readers interested strictly in the scalar/commutative setting may proceed directly to Subsection \ref{sec:sep} (or Section \ref{sec:reg}) without loss of continuity. Subsection \ref{sec:sep} proves a Hahn-Banach-type separation theorem for compact partially convex sets. Section \ref{sec:reg} investigates the geometry of partially convex sets and formally introduces the notion of regularity, which serves as the topological foundation for our duality. Within this section, Theorem \ref{thm:ctsXY} establishes a unique representation for continuous partially affine functions. Furthermore, Subsection \ref{sec:approx} provides a Stone-Weierstrass-type result (Proposition \ref{prop:p-approx}), proving that partially affine polynomials are dense in the space of continuous partially affine functions on a compact regular partially convex set. Section \ref{sec:caff} introduces the free order unit module, the abstract algebraic structure corresponding to the space of continuous partially affine functions. We begin by considering the topological properties of these modules, subsequently proving that $\Caff$ serves as a canonical model for such a structure (Proposition \ref{prop:axiomok}). In Section \ref{sec:state}, we construct the partially convex state space associated with a free order unit module and verify regularity of its coordinate realization (Proposition \ref{prop:Sreg}). Finally, Section \ref{sec:rep} provides the proof of Theorem \ref{thm:dual}.\looseness=-1

\subsection*{Acknowledgements}
This work was partially developed during a visit to Saarland University via the T4EU chair programme; I am grateful for their hospitality and support. Special thanks are due to Michael Hartz for insightful conversations regarding the properties of the function theoretic dual, which were crucial for the progress of this paper. I am also greatly appreciative of the helpful comments on the earlier versions of the manuscript provided by Igor Klep and Scott McCullough. Finally, I thank the anonymous referee for their careful reading and valuable suggestions.

\section{$\Gamma$-convexity and partial convexity}\label{sec:gamma}
In this section, we briefly situate partial convexity within the broader context of noncommutative $\Gamma$-convexity, providing the geometric motivation for our definitions. The remainder of the paper from Subsection \ref{sec:sep} onward focuses entirely on the commutative (scalar) setting.
%In this section, we formally define $\Gamma$-convexity and demonstrate how the notion of partial convexity arises naturally as a specific instance within this general theory.

Let $\Gamma=(\gamma_1,\dots,\gamma_r)$ be a tuple  of symmetric noncommutative polynomials in the self-adjoint variables $x_1,\ldots,x_g$
with $\gamma_j=x_j$ for $1\le j\le g\le r$. Denote by $\mathbb{S}_n$ the space of all complex self-adjoint matrices of size $n \times n$ and for $g \in \mathbb{N},$ let $\mathbb{S}^g = \bigcup_n \mathbb{S}^g_n.$

\begin{definition}\label{def gamma}
	(a) A pair $(X,V)$, where $X\in \mathbb{S}^g_n$
	and $V:\C^m\to \C^n$ is an isometry, is a \textbf{$\Gamma$-pair} if it satisfies
	\[
	V^* \Gamma(X)V= \Gamma(V^*XV).
	\]
	(b) A graded set $K = (K_n)_n \subseteq \mathbb{S}^g$ with
	$K_n \subseteq  \mathbb{S}^g_n$ for each $n$  is a
	\textbf{free set} if it is closed with respect to direct sums,
	simultaneous  unitary conjugations, and restrictions to reducing 
	subspaces.\footnote{
		This means that if  $X\in K_n$ and $Y\in K_m$, then $X\oplus Y\in K_{n+m};$
		if $U \in \mathbb{M}_n$ is unitary, then $U^*XU=(U^*X_1 U,\dots,U^*X_g U)\in K_n$;
		and if $\mathcal{S} \subset \C^n$ is a $k$-dimensional  reducing subspace for $X$, 
		then  $X|_{\mathcal{S}}\in K_k.$}
		
	\noindent
	(c) A set $\textbf{$K$} \subseteq \mathbb{S}^g$ is a \textbf{$\Gamma$-convex set} if it is free and if 
	\[
	X\in \textbf{$K$} \, \text{ and } \, (X,V) \text{ $\Gamma$-pair } \, \implies V^*XV\in \textbf{$K$}.
	\]
	%	
	%	(d) The \textbf{$\Gat$-convex hull} of a free set $\textbf{$K$}$ is the smallest $\Gat$-convex set that contains $\textbf{$K$}.$ It is obtained as the intersection of all $\Gat$-convex sets containing $\textbf{$K$}$ and we
	%	denote it  by $\Gamma\text{-conv}(\textbf{$K$}).$ 
\end{definition}

We use the abbreviation $y^2$-convex ($y^2$-pair) to refer to a $\Gamma$-convex set ($\Gamma$-pair) with $\Gamma = (x,y,y^2)$ (here both $x$ and $y$ can be tuples).
 By \cite{JKMMP21}, a tuple $((X,Y),V)$ is a $y^2$-pair if and only if the range of $V$ reduces $Y.$  
 
 It turns out that the notion of $y^2$-convexity coincides with partial convexity: the following is \cite[Proposition 4.1]{JKMMP21}.

\begin{proposition}\label{prop:y^2}
	A free set $\mathcal{S} = (S_n)_n \subseteq \mathbb{S}^g \times \mathbb{S}^h$ is $y^2$-convex if and only if each $S_n$ is partially convex in $x.$\looseness=-1
\end{proposition}

In the setting of matrix convexity, the natural morphisms are \textit{matrix affine maps}, which preserve direct sums and unitary conjugation. For $y^2$-convex sets, this notion is adapted as follows.

\begin{definition}\label{def:y^2map}
	Let $K = (K_n)_n$ be a $y^2$-convex set and $W$ a vector space. A family of maps $\Phi = (\Phi_n)_n,$ where $\Phi_n: K_n \to M_n(W),$ is a \textbf{$y^2$-affine map} if for any $y^2$-pair $((X,Y),V)$ with $(X,Y) \in K_k$ and $V \in \mathbb{M}_{k,r},$ 
	$$
	\Phi_r(V^*(X,Y)V) = V^* \Phi_k(X,Y) V.
	$$
\end{definition}

Any noncommutative polynomial $p$ can be evaluated at matrices of all sizes. We use the same notation $p$ for all evaluation maps $p_n: K_n \to \mathbb{M}_n,$ sending $(X,Y) \mapsto p(X,Y).$ So $p$ is considered a $y^2$-affine map if $p(V^*(X,Y)V) = V^*p(X,Y)V$ for all $y^2$-pairs $((X,Y),V).$

The following proposition motivates the definition of partially affine polynomials as morphisms of partially convex sets. The proof goes along the lines of the proof of Proposition \ref{prop:y^2}.

\begin{proposition}
	Let $K = (K_n)_n \subseteq \mathbb{S}^g$ be a $y^2$-convex set. A noncommutative polynomial $p$ is a $y^2$-affine map on $K$ if and only if it is affine in $x$ for any fixed $y.$
\end{proposition}

\begin{proof}
	To prove $(\Rightarrow)$ let $(X_1,Y), (X_2,Y) \in K_n.$ We want to prove that for any $0\leq t \leq 1,$
	\begin{equation*}
		p(tX_1 + (1-t)X_2,Y) = tp(X_1,Y) + (1-t)p(X_2,Y).
	\end{equation*}
	For $0\leq t \leq 1,$ let $V = (\sqrt{t} I_n \  \sqrt{1-t}I_n)^*.$ Note that if $(Z,W) = (X_1,Y) \oplus (X_2,Y) = (X_1 \oplus X_2, Y \oplus Y),$ then $((Z,W),V)$ is a $y^2$-pair and 
	$$
	V^* (Z,W) V = (tX_1 + (1-t)X_2,Y).
	$$
	Then
	\begin{align*}
		p(tX_1 + (1-t)X_2,Y) & = p(V^* (Z,W) V) = V^* p(Z,W) V \\
		& = V^* (p(X_1, Y) \oplus p(X_2,Y)) V \\
		& = tp(X_1,Y) + (1-t)p(X_2,Y).
	\end{align*}
	
	To see $(\Leftarrow),$ let $(X,Y) \in K_{n+m}$ and $V \in \mathbb{M}_{n+m,n}$ be an isometry such that $((X,Y),V)$ is a $y^2$-pair. Hence, with respect to the range of $V,$
	$$
	(X,Y) =  \left(
	\begin{pmatrix}
		X_{11} & X_{12} \\
		X_{21} & X_{22}
	\end{pmatrix},
	\begin{pmatrix}
		Y_1 & 0 \\
		0 & Y_2
	\end{pmatrix}
	\right).
	$$
	Let $U = I_n \oplus -I_m.$ Then $U^* (X,Y)U \in K_{n+m}$ and
	$$
	(\tilde{X},\tilde{Y}) = \frac12 (U^* (X,Y)U + (X,Y)) =  (X_{11} \oplus X_{22}, Y_1 \oplus Y_2) \in K_{n+m}.
	$$
	Now $p(\tilde{X},\tilde{Y}) = p(X_{11}, Y_1)\oplus p(X_{22},Y_2)$ and so
	\begin{align*}
		p(V^*(X,Y)V) =  p(X_{11}, Y_1) =  V^* p(\tilde{X},\tilde{Y})V = 
		\frac12 V^*p(U^*(X,Y)U)V + \frac12 V^*p(X,Y)V = V^* p(X,Y) V,
	\end{align*}
	since $V^*p(U^*(X,Y)U)V = V^*U^*p(X,Y)UV = V^* p(X,Y) V$ by definition of $U.$
\end{proof}

Motivated by the connection between $y^2$-convexity and partial convexity, we henceforth focus on the commutative (level $1$) case.
% We move on to establish the two parties in the duality theory for these sets: the category of compact regular partially convex sets and their functional counterparts, the free order unit modules.

\subsection{A separation result for partially convex sets}\label{sec:sep}
Central to classical convex analysis is the separation of a point from a closed convex set (the Hahn-Banach separation theorem). We demonstrate that this principle extends to compact partially convex sets, where separation is achieved by partially affine polynomials.
%For matrix convex sets, such a separation is performed via a linear pencil, i.e., a matrix valued affine polynomial. An analogous result holds for $\Gamma$-convex sets

\begin{theorem}\label{prop:sep}
	Let $K \subseteq \R^{n+m}$ be a compact partially convex set in x and $z \notin K.$ Then there exists a $p \in \Paff$ such that $p(w) \geq 0$ for all $w \in K$ and $p(z)<0.$
\end{theorem}

\begin{proof}
	Let $z = (x_z, y_z) \notin K$. Since $K$ is compact, the slice $K_{y_z} = \{x \in \mathbb{R}^n \mid (x, y_z) \in K\}$ is a compact (and, by assumption, convex) subset of $\mathbb{R}^n.$ Clearly, $x_z \notin K_{y_z}$.

	If $K_{y_z}$ is non-empty, then the Hahn-Banach separation theorem guarantees the existence of a vector $v \in \mathbb{R}^n$ and scalars $c,\delta \in \mathbb{R}$ with $\delta > 0$ such that
	$\langle v ,x \rangle + c \geq \delta > 0$ for all $x \in K_{y_z}$ but  $\langle v,  x_z \rangle + c < 0.$ Now set
	\[
	p(x,y) = \langle v ,x \rangle + c + M \|y - y_z\|^2,
	\]
	where $M > 0$ is a constant to be determined. Clearly, $p \in \Paff$ and $p(x_z, y_z) = \langle v,  x_z \rangle + c < 0$. 
	
	We must show there exists an $M$ such that $p(x,y) \geq 0$ for all $(x,y) \in K$. Assume for contradiction that no such $M$ exists. Then for every positive integer $k$ there exists a point $(x_k, y_k) \in K$ such that
	\[
	\langle v, x_k\rangle + c + k \|y_k - y_z\|^2 < 0.
	\]
	Since $K$ is compact, the function $f(x,y) =\langle v ,x \rangle + c$ attains a minimum on $K$, say $-B$ (where $B > 0$). Hence,
	\[
	-B + k \|y_k - y_z\|^2 < 0 \implies \|y_k - y_z\|^2 < \frac{B}{k}.
	\]
	Thus the $y_k$ converge to $y_z.$
	Again, since $K$ is compact, the sequence $(x_k, y_k)$ has a convergent subsequence that converges to some $(x^*, y^*) \in K$. By the above, we have $y^* = y_z$ and thus $x^* \in K_{y_z}$. 
	
	Now $\langle v, x^*\rangle + c \geq \delta > 0$. However, taking the limit of the strict inequality $$\langle v , x_k \rangle + c + k \|y_k - y_z\|^2 < 0$$ requires that $\langle v, x^*\rangle + c \leq 0,$ which is a contradiction. Thus, a sufficiently large $M$ must exist.

	On the other hand, if the slice $K_{y_z}$ is empty, no point in $K$ has $y = y_z$. In this case define
	\[
	p(x,y) = -1 + M \|y - y_z\|^2.
	\]
	Now $p \in \Paff$ and $p(x_z, y_z) = -1 < 0$. Because $K$ is compact and does not intersect the hyperplane $y = y_z$, we have
	$$\min_{(x,y) \in K} \|y - y_z\|^2 = \gamma > 0.$$
	By choosing $M \geq \frac{1}{\gamma}$, we ensure $p(x,y) \geq -1 + M\gamma \geq 0$ for all $(x,y) \in K$.
\end{proof}

\begin{remark}
	Theorem \ref{prop:sep} recovers the scalar case of \cite[Theorem 2.4]{JKMMP21}. The latter only assumes $K$ is closed, however, it uses the additional assumption that the convex hull of $\Gamma(K)$ is closed, where $\Gamma(x,y)=(x,y,y^2).$ 
\end{remark}

\begin{example}
	The separation in Theorem \ref{prop:sep} is not always feasible if the set $K$ is not bounded.
	Indeed, let $n=m=1$ and let
	\[
	K = \{(x, y) \in \mathbb{R}^2 \mid x \geq -e^{y^2}\}.
	\]
	This set is clearly a closed partially convex set in the $x$ variable.
	%is closed, and for any fixed $y$, the section $K_y = [-e^{y^2}, \infty)$ is convex. 
	Setting $z = (-2, 0)$ we have $z \notin K$.
	
	Assume, for the sake of contradiction, that there exists a polynomial $p \in \Paff$ such that $p(x, y) \geq 0$ for all $(x, y) \in K$, and $p(-2, 0) < 0$.
	The polynomial $p$ is of the form
	\[
	p(x, y) = v(y)x + c(y),
	\]
	where $v(y)$ and $c(y)$ are polynomials in $y$.
	First, observe that for any fixed $y$, as $x \to \infty$, the point $(x, y)$ is eventually in $K$. So for $p(x, y) \geq 0$ to hold for all arbitrarily large $x$, we must have $v(y) \geq 0$ for all $y \in \mathbb{R}$.
	Next, at the boundary of $K,$ we have
	\[
	v(y)(-e^{y^2}) + c(y) \geq 0 \implies c(y) \geq v(y)e^{y^2},
	\]
	which must hold for all $y \in \mathbb{R}$.
	Now at $z=(-2,0),$ we have by assumption that 
	\[
	p(-2, 0) = -2v(0) + c(0) < 0 \implies c(0) < 2v(0).
	\]
	If $v(y)$ were the zero polynomial, then we would require $c(y) \geq 0$ for all $y$. This would mean $c(0) \geq 0$, which contradicts $c(0) < 2v(0) = 0$. Thus, $v(y)$ is not the zero polynomial.
	
	Since $v(y)$ is a non-zero polynomial and $v(y) \geq 0$ for all $y$, it is strictly positive for sufficiently large $|y|$. Consequently, as $|y| \to \infty$, the term $v(y)e^{y^2}$ grows at least exponentially. This contradicts the requirement that $c(y) \geq v(y)e^{y^2}$ for all $y$ as $c(y)$ is a polynomial. Therefore, no such polynomial $p$ exists.
\end{example}

\begin{remark}
	The separation in Theorem \ref{prop:sep} cannot be generalized to arbitrary closed partially convex sets, even if we relax the separating class to include all continuous partially affine functions. The next proposition clarifies the nature of this limitation (the additional assumptions on the set $K$ in this proposition are closely related to the notion of regularity introduced in Section \ref{sec:reg}).
\end{remark}

\begin{definition}
	Let $\Phi: y \mapsto K_y$ be a set-valued map, where $y \in Y \subseteq \R^m$ and $K_y \subseteq \R^n$ for all $y \in Y.$\looseness=-1
	\begin{enumerate}[(a)]
		\item The map $\Phi$ is \textbf{lower hemicontinuous (LHC)} at $y$ if for every sequence $(y_k) \subset \mathbb{R}^m$ with $y_k \to y$, and for every $x \in K_y$, there exists a sequence $(x_k)$ such that
		\[
		x_k \in K_{y_k} \text{ for all } k, \quad \text{and} \quad x_k \to x.
		\]
		\item Assume $\Phi$ is compact-valued, meaning that each $K_y$ is compact. Then $\Phi$ is \textbf{upper hemicontinuous (UHC)} at $y$ if for every sequence $(y_k) \subset \mathbb{R}^m$ with $y_k \to y$, and every sequence $(x_k)$ such that $x_k \in K_{y_k}$ for all $k$, there exists a subsequence $(x_{k_j})$ and a point $x \in K_y$ such that
		\[
		x_{k_j} \to x \in K_y.
		\]
	\end{enumerate}
\end{definition}

\begin{proposition}\label{prop:sep2}
	Let $K \subseteq \mathbb{R}^{n+m}$ be a closed partially convex set in $x$ such that $\pi_2(K)$ is closed. Assume that for every $y \in \pi_2(K)$, the slice $K_y = \{x \in \mathbb{R}^n \mid (x, y) \in K\}$ is compact. Assume further that the set-valued mapping $y \mapsto K_y$ is continuous (both upper and lower hemicontinuous) on $\pi_2(K)$. If $z = (x_z, y_z) \notin K$, then there exists a continuous function $f: \mathbb{R}^{n+m} \to \mathbb{R}$ that is affine in $x$ for every fixed $y$, such that $f(x, y) \ge 0$ for all $(x, y) \in K$ and $f(x_z, y_z) < 0$.
\end{proposition}

\begin{proof}
	Let $z = (x_z, y_z) \notin K$.
	First assume $y_z \notin \pi_2(K)$. Recall that the projection $\pi_2(K)$ is assumed closed. Thus, the distance from $y_z$ to $\pi_2(K)$ is strictly positive, say $d > 0$. We can define the function 
	$$f(x,y) = \|y - y_z\|^2 - \frac{d^2}{2}.$$
	This function is constant with respect to $x$ (and thus trivially affine in $x$). For any $(x,y) \in K$, we have $\|y - y_z\| \ge d$, so $f(x,y) \ge \frac{d^2}{2} > 0$. However, at $z$, $f(x_z, y_z) = 0 - \frac{d^2}{2} < 0$. 
	
	Now assume $y_z \in \pi_2(K)$.
	Since $z \notin K$, we have $x_z \notin K_{y_z}$. By assumption, the slice $K_{y_z}$ is a compact convex subset of $\mathbb{R}^n$. By the strict Hahn-Banach separation theorem, there exists a vector $v \in \mathbb{R}^n$ and a scalar $c \in \mathbb{R}$ such that 
	%the affine function $F(x) = \langle v, x \rangle + c$ satisfies:
	$$F(x) = \langle v, x \rangle + c \ge \delta > 0 \quad \text{for all } x \in K_{y_z}$$
	and
	$$F(x_z) < 0.$$
	
	The function $F$ can be extended to $\mathbb{R}^{n+m}$ by setting $F(x, y) = \langle v, x \rangle + c$. 
	%To account for the variation of $K_y$ as $y$ changes, we 
	Now define the minimum value function $m: \pi_2(K) \to \mathbb{R}$ by
	$$m(y) = \min_{x \in K_y} F(x, y).$$
	Since $K_y$ is compact for all $y$, $m(y)$ is well-defined and finite. Next, because $F$ is continuous and the mapping $y \mapsto K_y$ is continuous, Berge's maximum theorem \cite[Theorem 17.31]{AB06} guarantees that $m(y)$ is a continuous function on $\pi_2(K).$ By the Tietze extension theorem, we can continuously extend $m(y)$ to all of $\mathbb{R}^m$.\looseness=-1
	
	Let  $\mu: \mathbb{R}^m \to \mathbb{R}$ be the correction term defined as
	$$\mu(y) = \max(0, -m(y)).$$
	Finally, define the separating function $f(x, y)$ as
	$$f(x, y) = F(x, y) + \mu(y).$$
	For any fixed $y$,  $\mu(y)$ is a constant, so $x \mapsto \langle v, x \rangle + c + \mu(y)$ is affine. Thus $f \in \Caff.$ 
	
	It remains to check the separation properties of $f.$
	At $y = y_z$, we have $m(y_z) = \min_{x \in K_{y_z}} F(x, y_z) \ge \delta > 0$. Therefore, $\mu(y_z) = \max(0, -m(y_z)) = 0$ and
	$$f(x_z, y_z) = F(x_z, y_z) + 0 = \langle v, x_z \rangle + c < 0.$$
	For any $(x, y) \in K$, we have $F(x, y) \ge m(y)$ by the definition of $m(y)$. Therefore
	$$f(x, y) = F(x, y) + \max(0, -m(y)) \ge m(y) + \max(0, -m(y)) \ge 0.$$
	Thus, $f(x, y)$ is a continuous, partially affine function that strictly separates $z$ from $K.$
\end{proof}

\section{Regularity of partially convex sets}\label{sec:reg}

Building on the separation results established in Subsection \ref{sec:sep}, this section introduces regularity conditions that ensure partially convex sets are topologically well-behaved. These assumptions imply a unique representation of continuous partially affine functions on \textit{regular} partially convex sets.\looseness=-1

Let $K \subseteq \R^{n+m}$ be a partially convex set.
	In general, one cannot assume that any continuous partially affine function $f$ on $K$  can be written as
	\[
	f(x,y) = c_0(y) + \sum_{i=1}^n c_i(y) x_i
	\]
	where $c_i(y)$ are continuous functions on the projection $\pi_2(K)$.
	This is because it is not generally possible to extend $f$ to a continuous partially affine function $F: \pi_1(K) \times \pi_2(K) \to \R.$

	\begin{example}\label{ex:not reg}
	Let $n = m = 1$ and let $K \subset \mathbb{R}^2$ be defined as
	\[
	K = \left\{ (x, y) \in \mathbb{R}^2 \mid y \in [0, 1], \ |x| \le y \right\}.
	\]
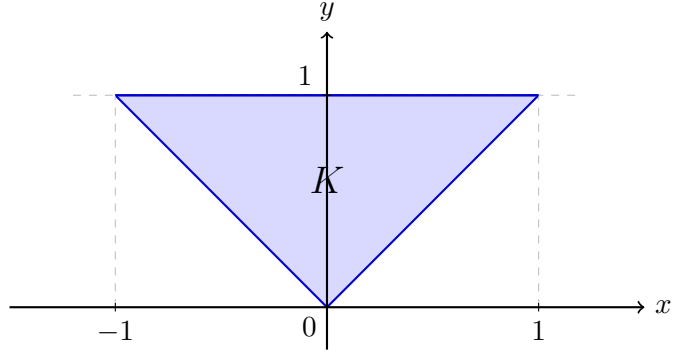
			\begin{center}
			\begin{tikzpicture}[scale=2.8]
				\fill[blue!30, opacity=0.5] 
				(0,0) 
				-- plot[domain=0:1, variable=\y, samples=100] ({\y}, \y) 
				-- ({-1}, 1) 
				-- plot[domain=1:0, variable=\y, samples=100] ({-\y}, \y) 
				-- cycle;
				
				% Draw the boundary lines thick
				\draw[thick, blue!80!black] 
				(0,0) 
				-- plot[domain=0:1, variable=\y, samples=100] ({\y}, \y);
				
				\draw[thick, blue!80!black] 
				(0,0) 
				-- plot[domain=0:1, variable=\y, samples=100] ({-\y}, \y);
				
				\draw[thick, blue!80!black] (-1,1) -- (1,1);
				
				% Draw Grid/Guides (Optional, helps visualize scale)
				\draw[dashed, gray!50] (-1,0) -- (-1,1);
				\draw[dashed, gray!50] (1,0) -- (1,1);
				\draw[dashed, gray!50] (-1.2,1) -- (1.2,1);
				\draw[thick, blue!80!black] (-1,1) -- (1,1);
				
				% Draw Axes
				\draw[->, thick] (-1.5,0) -- (1.5,0) node[right] {$x$};
				\draw[->, thick] (0,-0.2) -- (0,1.3) node[above] {$y$};
				
				% Axis Ticks and Labels
				\draw (1,0.02) -- (1,-0.02) node[below] {$1$};
				\draw (-1,0.02) -- (-1,-0.02) node[below] {$-1$};
				\draw (0.02,1) -- (-0.02,1) node[above left] {$1$};
				\node[below left] at (0,0) {$0$};
				
				% Label the set
				\node at (0, 0.6) {\Large $K$};	
			\end{tikzpicture}
			\end{center}
		Then $K$ is closed and partially convex (the slices $K_y = [-y, y]$ are closed intervals). Note that $\pi_1(K) = [-1, 1]$ and $\pi_2(K) = [0, 1]$. Define $f : K \to \mathbb{R}$ as
		\[
		f(x, y) = \begin{cases} 
			\frac{x}{\sqrt{y}} & \text{if } y \in (0, 1] \\
			0 & \text{if } y = 0
		\end{cases}
		\]
		This function is clearly affine in $x$ for each fixed $y.$ 
		%(for any fixed $y > 0$, $f(x, y) = y^{-1/2}x$, and for $y=0$, the domain is the point $\{0\}$).
		 Moreover, the function $f$ is continuous on $K$. To see continuity at $(0, 0)$, note that for any $(x,y) \in K$, we have $|x| \le y$. Thus:
		\[
		|f(x, y)| = \left| \frac{x}{\sqrt{y}} \right| \le \frac{y}{\sqrt{y}} = \sqrt{y}.
		\]
		Hence, as $(x, y) \to (0, 0)$ inside $K$, $|f(x, y)| \to 0 = f(0, 0)$.
		
		Assume towards reaching a contradiction that there exists an extension $F : \pi_1(K) \times \pi_2(K) \to \mathbb{R}$ that is continuous and affine in $x$. By definition, $F(x, y)$ must match $f(x, y)$ on the interval $[-y, y]$ for any $y > 0$. Since affine functions on intervals with non-empty interior are uniquely determined, we have
		\[
		F(x, y) = \frac{x}{\sqrt{y}} \quad \text{for all } x \in [-1, 1], \ y > 0.
		\]
		Now consider the point $(1, 0) \in \pi_1(K) \times \pi_2(K)$. When approaching this point along the line $x = 1$, we get
		\[
		\lim_{y \to 0^+} F(1, y) = \lim_{y \to 0^+} \frac{1}{\sqrt{y}} = \infty,
		\]
		so $F$ cannot be defined at $(1, 0)$ in a way that preserves continuity. Thus no such extension exists.
	\end{example}

\begin{definition}\label{def:reg}
	Let $K \subset \R^{n+m}$ be a closed set. Then $K$ is called a \textbf{regular partially convex set} if the following holds:
	\begin{enumerate}
		\item $K$ is partially convex (i.e., $K_y$ is convex for all $y$),
		\item  the interior of $K_y$ in $\R^n$ is nonempty for all $y \in \pi_2(K),$
		\item the set-valued mapping $y \mapsto \operatorname{int}(K_y)$ is lower hemicontinuous (LHC).
	\end{enumerate}
\end{definition}

\begin{remark} \label{rem:lhc}
		%Recall that a set-valued map $\Phi : Y \to \Omega$ is LHC at $y_0$ if for every point $z_0 \in \Phi(y_0)$ and every sequence $y_n \to y_0$, there exists a sequence of points $z_n$ such that $z_n \in \Phi(y_n)$ for all $n$, and $z_n \to z_0$.
Roughly speaking, LHC guarantees that the set does not suddenly shrink or collapse as the parameter $y$ changes. 
 If each $K_y \subseteq \R^n$ is a convex set with non-empty interior, then the map $y \mapsto K_y$ is LHC if and only if the map $y \mapsto \text{int}(K_y)$ is LHC.
 
 This follows from two facts.
 First, an equivalent definition of lower hemicontinuity of a set-valued map $z \mapsto \Phi(z)$ at a point $z$ demands that for every open set $V \subseteq X$ such that $\Phi(z) \cap V \neq \emptyset$, there exists a neighborhood $U$ of $z$ such that $\Phi(z') \cap V \neq \emptyset$ for all $z' \in U.$
 
 Secondly, any convex set $K \subseteq \R^n$ with a non-empty interior 
 satisfies
 \[
 K \cap V \neq \emptyset \iff \text{int}(K) \cap V \neq \emptyset
 \]
 for any open set $V.$
Indeed, the implication $\impliedby$ is trivial and the other implication $\implies$ follows from the fact that $K$ is contained in the closure of its interior, that is, $K \subseteq \overline{\text{int}(K)}$. 

% the other implication $\implies$, suppose $K \cap V \neq \emptyset$ and let $x \in K \cap V$. Because $K$ is convex with a non-empty interior, $K$ is contained in the closure of its interior; that is, $K \subseteq \overline{\text{int}(K)}$. Therefore, $x \in \overline{\text{int}(K)}$. Since $V$ is an open neighborhood containing $x$, it must intersect $\text{int}(K)$. \hfill $\blacksquare$
\end{remark}

\begin{remark}
	The set $K$ from Example \ref{ex:not reg} is a partially convex set, which is not regular as it fails condition $2$ in Definition \ref{def:reg}.
	
The figure below further illustrates the conditions in Definition \ref{def:reg}.
%	The set $L$ on the left is a regular partially convex set. Indeed, for any $y \in [0, 1]$ we have $L_y = [-1, 1],$ so $\text{int}(L_y) = (-1, 1) \neq \emptyset.$ Moreover, the map $y \mapsto (-1, 1)$ is LHC as it is constant: for any $x \in (-1, 1)$ and sequence $y_n \to y$, we can simply choose $x_n = x \in \text{int}(L_{y_n})$ to satisfy $x_n \to x$.
 The set $L$ on the left (analogous to the set in Figure~\ref{fig1}) is a regular partially convex set. Indeed, for each $y \in [0, 1]$, the slice is $L_y = [-\sqrt{(y^2+1)/2}, \sqrt{(y^2+1)/2}]$, so $\operatorname{int}(L_y) = (-\sqrt{(y^2+1)/2}, \sqrt{(y^2+1)/2}) \neq \emptyset$. Moreover, the map $y \mapsto \operatorname{int}(L_y)$ is LHC by continuity of the boundary curves $y \mapsto \pm\sqrt{(y^2+1)/2}$.
	
	The set $M = ([-1, 1] \times (0, 1]) \cup ([-2, 2] \times \{0\})$ on the right is partially convex and satisfies conditions $1$ and $2$ of Definition \ref{def:reg}. However, it fails to satisfy condition $3$ as the map $y \mapsto \operatorname{int}(M_y)$ is not LHC.
 To see this, consider the point $x_0 \in \text{int}(M_0)$ as marked on the figure below. Consider the sequence $y_n = 1/n \to 0.$ Then there is no sequence of points $x_n \in (-1, 1) =  \text{int}(M_{y_n})$ such that $x_n \to x_0 > 1$.

	\begin{center}
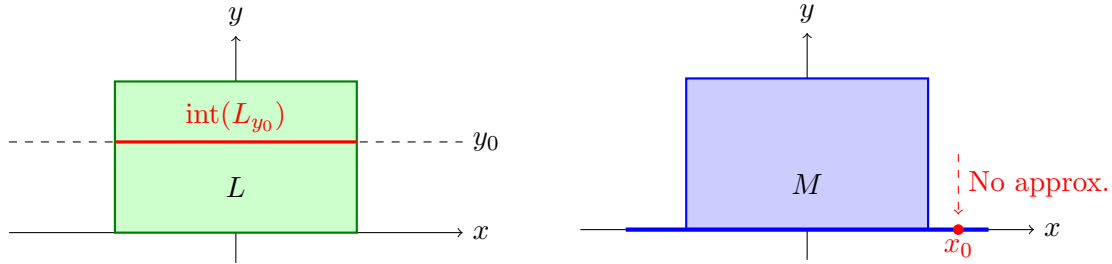
		\begin{minipage}{0.45\textwidth}
			\begin{tikzpicture}[scale=2]
				% Axes
				\draw[->] (-1.5,0) -- (1.5,0) node[right] {$x$};
				\draw[->] (0,-0.2) -- (0,1.3) node[above] {$y$};
				
				% Curved, non-convex regular partially convex set (from 2x^2 <= y^2 + 1)
				\fill[green!20] plot[domain=0:1, samples=50] ({sqrt((\x*\x + 1)/2)}, \x) -- 
				plot[domain=1:0, samples=50] ({-sqrt((\x*\x + 1)/2)}, \x) -- cycle;
				\draw[thick, green!50!black] plot[domain=0:1, samples=50] ({sqrt((\x*\x + 1)/2)}, \x) -- 
				plot[domain=1:0, samples=50] ({-sqrt((\x*\x + 1)/2)}, \x) -- cycle;
				
				% Slice at y0 = 0.6
				\pgfmathsetmacro{\yzero}{0.6}
				\pgfmathsetmacro{\xval}{sqrt((\yzero*\yzero + 1)/2)}
				\draw[dashed] (-1.5, \yzero) -- (1.5, \yzero) node[right] {$y_0$};
				\draw[very thick, red] (-\xval, \yzero) -- (\xval, \yzero);
				\node[red, above] at (0, \yzero) {$\operatorname{int}(L_{y_0})$};
				\node at (0, 0.25) {$L$};
			\end{tikzpicture}
		\end{minipage}
		%\hfill
		% FIGURE 2: LHC FAILURE
		\begin{minipage}{0.45\textwidth}
			\begin{tikzpicture}[scale=2]
				% Axes
				\draw[->] (-1.5,0) -- (1.5,0) node[right] {$x$};
				\draw[->] (0,-0.2) -- (0,1.3) node[above] {$y$};
				
				% The Body (y > 0)
				\fill[blue!20] (-0.8,0) rectangle (0.8,1);
				\draw[thick, blue] (-0.8,0) -- (-0.8,1) -- (0.8,1) -- (0.8,0);
				
				% The Base (y = 0)
				\draw[ultra thick, blue] (-1.2,0) -- (1.2,0);
				%\fill[blue] (-2,0) circle (0.03);
				%\fill[blue] (2,0) circle (0.03);
				
				% Illustration of failure
				\fill[red] (1, 0) circle (1pt);
				\node[red, below] at (1, 0) {$x_0$};
				
				\draw[->, red, dashed] (1, 0.5) -- (1, 0.1);
				\node[red, right] at (1, 0.3) {No approx.};
				\node at (0, 0.3) {$M$};
				
			%	\node at (0, -0.4) {\textbf{Not Regular}};
			%	\node at (0, -0.6) {(LHC Fails at $y=0$)};
			\end{tikzpicture}
		\end{minipage}
	\end{center}
\end{remark}

\begin{lemma}\label{lem:coeff_recovery}
	Let $K \subseteq \mathbb{R}^{n+m}$ be a regular partially convex set, let $Y = \pi_2(K)$, and let $f \in C_{\text{aff}}(K)$. For each $y_0 \in Y$, there exist a neighborhood $U \subseteq Y$ of $y_0$ and continuous local selections $s_0, s_1, \dots, s_n : U \to \mathbb{R}^n$ such that:
	\begin{enumerate}[label=\rm{(\roman*)}]
		\item For every $y \in U$, the points $\{s_0(y), \dots, s_n(y)\}$ are affinely independent.
		\item The matrix $M(y) \in M_{n+1}(\mathbb{R})$ with $k$-th row $[1, (s_k(y))^T]$ is invertible on $U$, with continuous inverse $y \mapsto M(y)^{-1}$.
		\item The unique coefficients $c(y) = (c_0(y), \dots, c_n(y))^T$ satisfying $f(x, y) = c_0(y) + \sum_{i=1}^n c_i(y)x_i$ on $K_y$ are given by
		\[
		c(y) = M(y)^{-1} F(y) \quad \text{for all } y \in U,
		\]
		where $F(y) = (f(s_0(y), y), \dots, f(s_n(y), y))^T$. In particular, $y \mapsto c(y)$ is continuous on $Y$.
	\end{enumerate}
	Furthermore, if $K$ is compact, there exists a constant $C > 0$ such that for all $f \in C_{\text{aff}}(K),$ 
	%and all $y \in Y$,
	%\[
	%\|c(y)\|_{\mathbb{R}^{n+1}} \le C \sup_{x \in K_y} |f(x, y)| \le C \|f\|_\infty.
	%\]
	\[
	\sup_{y \in Y} \|c(y)\|_{\mathbb{R}^{n+1}} \le C \|f\|_\infty, \quad \textnormal{and hence} \quad \max_{0 \le i \le n} \|c_i\|_{C(Y)} \le C \|f\|_\infty.
	\]
\end{lemma}

\begin{proof}
	Let $y_0 \in Y = \pi_2(K)$. Since $\operatorname{int}(K_{y_0}) \neq \emptyset$, there exist $n + 1$ affinely independent points $v_0, \dots, v_n \in \operatorname{int}(K_{y_0})$. 
	By regularity of $K$ (Definition~\ref{def:reg}) and Remark~\ref{rem:lhc}, the slice mapping $y \mapsto K_y$ has non-empty closed convex values in $\mathbb{R}^n$ and is lower hemicontinuous (LHC). By the Michael selection theorem \cite[Theorem~17.66]{AB06}, there exists an open neighborhood $U \subseteq Y$ of $y_0$ and continuous local selections $s_k : U \to \mathbb{R}^n$ such that $s_k(y) \in K_y$ for all $y \in U$ and $s_k(y_0) = v_k$ for each $k = 0, \dots, n$.\footnote{To obtain a selection $s_k$ with prescribed value  $s_k(y_0) = v_k$ the Michael selection theorem is applied to the LHC map defined by $\Psi_k(y_0) = \{v_k\}$ and $\Psi_k(y) = K_y$ for $y \neq y_0.$}

	For each $y \in U$, evaluating the affine function $f(\cdot, y)$ at the points $s_k(y)$ yields the linear system
	\[
	c_0(y) + \sum_{i=1}^n c_i(y) (s_k(y))_i = f(s_k(y), y), \quad k = 0, \dots, n,
	\]
	which in matrix form is $M(y) c(y) = F(y)$, where $M(y) \in M_{n+1}(\mathbb{R})$ has $k$-th row $[1, (s_k(y))^T]$ and $F(y) = (f(s_0(y), y), \dots, f(s_n(y), y))^T$.
	
	Because each $s_k$ and $f$ are continuous, the entries of $M(y)$ and $F(y)$ depend continuously on $y \in U$. At $y = y_0$, the points $s_k(y_0) = v_k$ are affinely independent, so $\det(M(y_0)) \neq 0$. By continuity of the determinant, shrinking $U$ ensures $\det(M(y)) \neq 0$ for all $y \in U$. In particular, the points $\{s_0(y), \dots, s_n(y)\}$ remain affinely independent on $U$, establishing (i) and (ii).
	
	Since matrix inversion is smooth on $\operatorname{GL}_{n+1}(\mathbb{R})$, the map $y \mapsto M(y)^{-1}$ is continuous on $U$. Thus, the unique coefficient vector on $U$ is given by
	\[
	c(y) = M(y)^{-1} F(y),
	\]
	which is continuous at $y_0$. Since $y_0 \in Y$ was arbitrary, $y \mapsto c(y)$ is continuous on all of $Y$, proving (iii).
	
	Now assume $K$ is compact, which implies $Y$ is compact. For each $y \in Y$, choose an open neighborhood $V(y)$ such that $\overline{V(y)} \subseteq U(y)$, where $U(y)$ is a neighborhood on which $M(y)^{-1}$ is well-defined and continuous. Since $\overline{V(y)}$ is compact, $\sup_{z \in \overline{V(y)}} \|M(z)^{-1}\| < \infty$. By compactness of $Y$, there exist finitely many points $y_1, \dots, y_p \in Y$ such that $Y \subseteq \bigcup_{j=1}^p V(y_j)$. 
	
	Setting $C_j = \sup_{z \in \overline{V(y_j)}} \|M(z)^{-1}\|$ and $C = \sqrt{n+1} \max_{1 \le j \le p} C_j < \infty$, for any $y \in Y$ there is some $j$ with $y \in V(y_j)$. Since $s_k(y) \in K_y$, we have $|f(s_k(y), y)| \le \sup_{x \in K_y} |f(x, y)|$, so $\|F(y)\|_{\mathbb{R}^{n+1}} \le \sqrt{n+1} \sup_{x \in K_y} |f(x, y)|$. Therefore,
	\[
	\|c(y)\|_{\mathbb{R}^{n+1}} \le \|M(y)^{-1}\| \|F(y)\|_{\mathbb{R}^{n+1}} \le C \sup_{x \in K_y} |f(x, y)| \le C \|f\|_\infty.
	\]
	Taking the supremum over all $y \in Y$ and noting that $|c_i(y)| \le \|c(y)\|_{\mathbb{R}^{n+1}}$, we obtain $$\max_{0 \le i \le n} \|c_i\|_{C(Y)} \le C \|f\|_\infty,$$ which completes the proof.
\end{proof}

	\begin{theorem}\label{thm:ctsXY}
		Let $K \subset \R^{n+m}$ be a regular partially convex set. If $f$ is a continuous partially affine function on $K,$ then there exist unique continuous functions $c_0, c_1, \dots, c_n: Y \to \R$ such that for all $(x,y) \in K,$
		\[
		f(x,y) = c_0(y) + \sum_{i=1}^n c_i(y)x_i.
		\]
	\end{theorem}
	
\begin{proof}
	Since $\operatorname{int}(K_y) \neq \emptyset$ for any $y,$ the domain of $f(\cdot, y)$ is full-dimensional\footnote{This means that the affine span of $K_y$ is $\R^n$ for all $y.$} in $\R^n$. Since $f(\cdot, y)$ is affine on $K_y$ and the representation of an affine function on a full-dimensional set is unique, the scalars $c_0(y), \dots, c_n(y)$ are uniquely defined for every $y$.  The continuity of the coefficient functions $c_0, c_1, \dots, c_n : Y \to \mathbb{R}$ follows immediately from Lemma~\ref{lem:coeff_recovery}.
\end{proof}

	\subsection{An approximation result}\label{sec:approx}
A Stone-Weierstrass-type approximation result for continuous partially affine functions on a regular partially convex set now easily follows from Theorem \ref{thm:ctsXY}. The proof is included for completeness.
	
	\begin{proposition}\label{prop:p-approx} Let $K \subseteq \R^{n+m}$ be a compact regular partially convex set.
		The completion of the space of partially affine polynomials $\Paff$ with respect to the supremum norm is the space of continuous partially affine functions $\Caff$. That is,
		\[ \overline{\Paff} = \Caff. \]
	\end{proposition}
	
	\begin{proof}
		To prove inclusion $\overline{\Paff} \subseteq \Caff,$
		let $f \in \overline{\Paff}$. Then there exists a sequence of polynomials $(p_k)_k$ in $\Paff$ such that $p_k \to f$ uniformly on $K$. We must show that $f \in \Caff$.
		
Clearly, $f$ is continuous as it is a uniform limit of continuous functions.
 To see that $f$ is affine in $x,$ fix an arbitrary $y_0 \in \pi_2(K),$ let $x_1, x_2 \in K_{y_0}$ and $t \in [0, 1]$. 
 Since the convergence $p_k \to f$ is uniform on all of $K$, it is also uniform on the slice $K_{y_0}$. Hence
			\begin{align*}
				f(tx_1 + (1-t)x_2, y_0) &= \lim_{k\to\infty} p_k(tx_1 + (1-t)x_2, y_0) \\
				&= \lim_{k\to\infty} \left( t p_k(x_1, y_0) + (1-t) p_k(x_2, y_0) \right) \\
				&= t \lim_{k\to\infty} p_k(x_1, y_0) + (1-t) \lim_{k\to\infty} p_k(x_2, y_0) \\
				&= t f(x_1, y_0) + (1-t) f(x_2, y_0).
			\end{align*}
			This holds for all $x_1, x_2 \in K_{y_0}$ and $t\in [0, 1]$, which proves that the function $x \mapsto f(x, y_0)$ is affine on $K_{y_0}$.
		Thus $f \in \Caff$.
		
		To prove inclusion $\Caff \subseteq \overline{\Paff},$ we show that any function in $\Caff$ can be uniformly approximated by a sequence of functions in $\Paff$.	Let $f \in \Caff$. By Theorem \ref{thm:ctsXY}, the function $f$ can be written as
		\[ f(x, y) = c_0(y) + \sum_{i=1}^n c_i(y) x_i  \]
		where the $c_i(y)$ are continuous functions on $\pi_2(K)$. Since $K$ is compact, its projection $\pi_2(K)$ is also a compact subset of $\R^m$.
		
		Let $\epsilon > 0$ be given. By the Weierstrass Theorem, each of the continuous functions $c_i(y)$ on $\pi_2(K)$ has a polynomial approximation. That is, there exist polynomials $p_i(y)$ such that for all $y \in \pi_2(K),$
		\[ |c_i(y) - p_i(y)| < \epsilon.\]
		Now the polynomial $p \in \Paff,$
		\[ p(x, y) =  p_0(y) + \sum_{i=1}^n p_i(y) x_i, \]
		satisfies 
		\begin{align*}
			|f(x, y) - p(x, y)| &
			\le |c_0(y) - p_0(y)| + \sum_{i=1}^n |c_i(y) - p_i(y)| |x_i|  \\
			&\le \epsilon + \sum_{i=1}^n \epsilon |x_i|.
		\end{align*}
		Since $K$ is compact, the coordinates $x_i$ are bounded on $K$. 
	We can thus find an arbitrarily good approximation of $f$ with polynomials from $\Paff$ showing that $f$ is in the closure of $\Paff$.
	\end{proof}

		\section{Abstract Structure of $\cC_{\rm{aff}}(K)$}\label{sec:caff}
	
	We now identify the axioms characterizing the abstract structure of the space of continuous partially affine functions.
		For a regular partially convex set $K \subseteq \R^{n+m},$ the space $\Caff$ is both an Archimedean order unit (AOU) space and a finitely generated free module over $\cC(Y),$ where $Y = \pi_2(K)$ is the projection of $K$ onto $\R^m.$ We refer the reader to \cite{Al} for the basics of AOU spaces.

 \begin{definition}\label{def:module}
		A \textbf{free order unit module of type $(n,m)$} is a set $A$ equipped with the following structure:
		\begin{enumerate}
			\item \textbf{AOU Space:} $A$ is an Archimedean Order Unit (AOU) space with order unit $u$ and positive cone $A^+$.
			
			\item \textbf{Module structure over its multiplicative domain:} The space $A$ is a module over its multiplicative domain
			\[
			\mathcal{M}_A = \{z \in A \mid \forall a \in A, \, z \cdot a \in A\},
			\]
			where multiplication is induced via the canonical Kadison embedding $A \hookrightarrow C(\mathcal{S}(A))$ (see Remark~\ref{rem:kadison_mult} below). 
			%We assume that $\mathcal{M}_A$, equipped with the norm inherited from $A$, is a unital commutative real $C^*$-algebra \CO{minimally} generated by $m$ elements, with unit $1_A$ acting as the identity on $A$. By the Gelfand--Naimark theorem, the character space of $\mathcal{M}_A$ embeds homeomorphically into $\mathbb{R}^m$ as a compact subset $Y \subseteq \mathbb{R}^m$, yielding an isometric $*$-isomorphism $\mathcal{M}_A \cong C(Y)$.
			 We assume that $\mathcal{M}_{A}$, equipped with the norm inherited from $A$, is a unital commutative real $C^*$-algebra with unit $1_A$, and that $\mathcal{M}_A$ is generated by $m$ elements as a unital $C^*$-algebra. By the Gelfand–Naimark theorem, the evaluation map $\chi \mapsto (\chi(g_1), \dots, \chi(g_m))$ embeds the character space $\Omega(\mathcal{M}_A)$ homeomorphically onto a compact subset $Y\subseteq\mathbb{R}^{m}$, yielding an isometric $*$-isomorphism $\mathcal{M}_{A}\cong \mathcal{C}(Y)$.

			\item \textbf{Finitely generated and free:} $A$ is a finitely generated, free module over $\cM_A$ of rank $n+1$. This means there exists a basis $\{e_0, e_1, \dots, e_n\}$ of $A$ such that every element $a \in A$ has a unique representation
			\[ a = \sum_{i=0}^n c_i e_i \quad \text{with } c_i \in \cM_A. \]
			%\CO{The choice of a basis is part of the free order unit module structure.}
			
			\item \textbf{Unit alignment:} The order unit $u$ of $A$ is aligned with the basis and the module structure, i.e., $u = e_0$.   
			
			%\item \textbf{Fiber-wise order:} 
			%For each $y \in Y$, let $J_y$ be the closed order ideal in $A$ generated by the maximal ideal $I_y = \{f \in \mathcal{M}_A \cong \cC(Y) \mid f(y)=0\}$. The quotient space $A_y = A/J_y$ is a finite-dimensional AOU space of dimension $n+1$. Furthermore, the global order structure is determined pointwise,
			%$$A^+ = \{a \in A \mid \forall y \in Y, [a]_y \in A_y^+\},$$
			%and $[u]_y$ is the local order unit $u_y$ of $A_y$ for all $y \in Y$.

%			\item \textbf{Continuity Condition:} 
%			The set-valued map $y \mapsto A_y^+$ is LHC.
%			%, meaning that for any $y \in Y,$ element $[a]_y \in A_y^+$ and sequence $y_k \to y,$ there is a sequence $([a_k]_{y_k})_k$ with $[a_k]_{y_k} \in A_{y_k}^+$ such that $[a_k]_{y_k} \to [a]_y.$\CO{ne rabi še enkrat definicije LHC}
%			Moreover, the map $y \mapsto A_y^+$ is UHC on bounded slices, meaning that for every $r>0,$ the set-valued map
%			$$
%			y \mapsto A_y^+ \cap \{a \ | \ \|a\| \leq r\}
%			$$
%			is UHC. 
			
			\item \textbf{Fiber-wise order:} For each $y \in Y$, let $J_y$ be the closed order ideal in $A$ generated by the maximal ideal $I_y = \{f \in M_A \cong C(Y) \mid f(y) = 0\}$. The quotient space $A_y = A/J_y$ is endowed with the quotient positive cone
			\[
			A_y^+ := \{[a]_y \in A_y \mid a \in A^+\},
			\]
			which makes $A_y$ a finite-dimensional AOU space of dimension $n+1$ with local order unit $u_y = [u]_y$. Furthermore, the global order structure is determined pointwise:
			\[
			a \in A^+ \iff \forall y \in Y, \, [a]_y \in A_y^+.
			\]
			
			\item \textbf{Continuity Condition (UHC on bounded slices):}
			%\CO{reformulate with no coordinates-is LHC automatic? dont need to prove it in prop 4.10} 
%			Under the coordinate identification of each fiber $A_y$ with $\mathbb{R}^{n+1}$ via the basis $\{[e_0]_y, \dots, [e_n]_y\}$, the set-valued map $y \mapsto A_y^+ \subseteq \mathbb{R}^{n+1}$ is LHC. Moreover, for every $r > 0$, the set-valued map
%			\[
%			y \mapsto A_y^+ \cap \{v \in A_y \mid \|v\|_y \le r\} \subseteq \mathbb{R}^{n+1}
%			\]
%			is UHC.
			%The family of positive cones $(A_y^+)_{y \in Y}$ is lower hemicontinuous (for every $y_k \to y_0$ and $v \in A_{y_0}^+$, there exist $v_k \in A_{y_k}^+$ such that $v_k \to v$ in the fiber bundle) and upper hemicontinuous on bounded slices (for every sequence $y_k \to y_0$ and $v_k \in A_{y_k}^+$ with $\sup_k \|v_k\|_{y_k} < \infty$, there exists a subsequence $v_{k_j}$ and an element $v_0 \in A_{y_0}^+$ such that $v_{k_j} \to v_0$ in the fiber bundle).
			For every sequence $y_k \to y_0$ in $Y$ and every sequence $v_k \in A_{y_k}^+$ satisfying $\sup_{k} \|v_k\|_{y_k} < \infty$, there exists a subsequence $(v_{k_j})$, an element $v_0 \in A_{y_0}^+$, and a global section $a \in A$ with $[a]_{y_0} = v_0$ such that 
			\[
			\lim_{j \to \infty} \|v_{k_j} - [a]_{y_{k_j}}\|_{y_{k_j}} = 0.
			\]
			(that is, $v_{k_j} \to v_0$ in the fiber bundle in the sense of Definition~\ref{def:fiber_conv} below).
			
			\item \textbf{Topological Regularity:} 
			For some (and hence any) module basis $\{e_{0}, e_{1}, \dots, e_{n}\}$ of $A$ with $e_0 = u$, the linear isomorphism $\Psi: A \to \mathcal{C}(Y)^{n+1}$ given by
			$$\Psi(a) = (c_{0}, c_{1}, \dots, c_{n}) \quad \text{ where } a = \sum_{i=0}^{n} c_{i} e_{i}$$
			is a homeomorphism between $A$ (with its order unit norm topology) and $\mathcal{C}(Y)^{n+1}$ (with the product supremum norm topology).  That is,
			 \[
			 \|a^{(k)} - a\| \to 0 \iff \forall i \in \{0, \dots, n\}, \ \sup_{y \in Y} |c_i^{(k)}(y) - c_i(y)| \to 0.
			 \]

%			\item \textbf{The Norm and Archimedean Property:}
%			\CO{Should all this follow from 5?}
%			This global norm definition is consistent with the supremum of the fiber norms,
%			\CO{Does this follow by 5?}
%			\[
%			\|a\| = \sup_{y \in Y} \left( \inf \{ \lambda > 0 \mid -\lambda \cdot 1_{A_y} \le_{A_y} [a]_y \le_{A_y} \lambda \cdot 1_{A_y} \} \right) = \sup_{y \in Y} \|[a]_y\|_{A_y}.
%			\]
%			
%			Moreover, for each $a,$ the function $y \mapsto \|[a]_y\|_{A_y}$ is continuous.

			%\item \textbf{Topology equivalence:} The norm topology of $A$ is equivalent to the topology of convergence within each fiber $A_y.$
		\end{enumerate}
	\end{definition}
	
		\begin{remark}\label{rem:kadison_mult} 
		Let $A$ be an AOU space. Then the Kadison representation theorem \cite{Al} asserts that $A$ is order-isomorphic to the space of continuous affine functions on the compact convex set $S(A)$ (the state space of $A$). Consequently, $A$ admits a canonical embedding into the commutative $C^*$-algebra $\cC(S(A))$, the space of all continuous functions on $S(A).$ Through this embedding, the natural pointwise multiplication in $\cC(S(A))$ induces a well-defined module action of the subalgebra $\cM_A \subseteq \cC(S(A))$ on $A$.
	\end{remark}

%		\begin{remark} Let $A$ be a free order unit module dual (as an AOU space) to a compact convex set $K.$ By definition, the multiplicative domain $\mathcal{M}_A$ is a subspace of $A$. Since $A$ can be canonically embedded into $\cC(K),$ the multiplicative domain $\mathcal{M}_A$ can be identified with a subset of $C(K).$ The module action of $\cM_A$ on $A$ is realized inherently via the pointwise multiplication of $C(K)$. 
%			%Therefore, the identification $\mathcal{M}_A \cong C(Y)$ naturally positions $C(Y)$ as a $C^*$-subalgebra of $C(K)$
%	\end{remark}
%	
		\begin{definition} \label{def:fiber_conv}
		Let $A$ be a free order unit module. Let $y_k \to y_0$ be a convergent sequence in $Y$. A sequence of elements $[a_k]_{y_k}$ in the fibers $A_{y_k}$ is said to \textbf{converge} to an element $[a_0]_{y_0} \in A_{y_0}$ if there exists a global section $a \in A$ such that $[a]_{y_0} = [a_0]_{y_0}$ and
		\[ \lim_{k \to \infty} \| [a_k]_{y_k} - [a]_{y_k} \|_{y_k} = 0. \]
		Here $\|\cdot\|_{y_k} = \|\cdot\|_{A_{y_k}}.$
	\end{definition}
	
		\begin{remark}\label{re:lhc1}
		Note that the map $y \mapsto A_y^+$ is automatically LHC under the bundle convergence of Definition~\ref{def:fiber_conv} (i.e., for every $y_k \to y_0$ and $v \in A_{y_0}^+$, there exist $v_k \in A_{y_k}^+$ such that $v_k \to v$). Indeed, for any $v \in A_{y_0}^+$, by Axiom~(5) there exists a global positive element $a \in A^+$ such that $[a]_{y_0} = v$. Setting $v_k = [a]_{y_k} \in A_{y_k}^+$ for any sequence $y_k \to y_0$, we have $\|v_k - [a]_{y_k}\|_{y_k} = 0 \to 0$, so $v_k \to v$ in the fiber bundle. Thus, only the upper hemicontinuity of bounded slices is required as an axiom.
	\end{remark}

\begin{definition}\label{def:module_iso}
	Let $A$ and $B$ be two free order unit modules of type $(n,m)$ with multiplicative domains $\mathcal{M}_A$ and $\mathcal{M}_B$, respectively. A map $\varphi : A \to B$ is an \textbf{isomorphism of free order unit modules} if:
	\begin{enumerate}
		\item $\varphi$ is a unital order isomorphism (i.e., a linear bijection with $\varphi(u_A) = u_B$ and $\varphi(A^+) = B^+$);
		\item the restriction $\rho = \varphi|_{\mathcal{M}_A}$ maps into $\mathcal{M}_B$ and is an isometric $*$-isomorphism of $C^*$-algebras;
		\item $\varphi$ preserves the module action along $\rho$, that is,
		\[
		\varphi(z \cdot a) = \rho(z) \cdot \varphi(a) \quad \text{for all } z \in \mathcal{M}_A \text{ and } a \in A.
		\]
	\end{enumerate}
\end{definition}
	
\begin{remark}\label{re:conv}
%	By Axiom (3), any element $a \in A$ has a unique representation
%	\[
%	a = \sum_{i=0}^n c_i e_i, \quad \text{where } c_i \in \cM_A \cong \cC(Y).
%	\]
%	An element $[a]_y$ of the fiber $A_y$ as defined in Axiom (5) is the result of evaluating the coefficient functions $c_i$ at the specific point $y,$
%	\[
%	[a]_y = \sum_{i=0}^n c_i(y) [e_i]_y.
%	\]
%	Here $c_i(y) \in \R$ and $[e_0]_y, \dots, [e_n]_y$ are the projections of the global basis vectors, which form a basis of $A_y$. 
	
	By Axiom~(3), every element $a\in A$ has a unique representation
	\[
	a=\sum_{i=0}^n c_i e_i,
	\qquad c_i\in \cM_A\cong \cC(Y).
	\]
	We first explain how this representation descends to the fibers. For $g\in \cM_A$ and $a\in A$ we claim that
	\begin{equation}\label{eq:fiber-module-action}
		[g\cdot a]_y=g(y)[a]_y,
		\qquad y\in Y.
	\end{equation}
	Indeed, put $h=g-g(y)1_A\in I_y$ 
	%Since $\cM_A$ is a real $C^*$-algebra, 
	and note that $|h|\in I_y\subseteq J_y$. Moreover, through the
	canonical Kadison embedding $A\hookrightarrow \cC(S(A))$ we have
	\[
	-\|a\|\,|h|
	\leq h\cdot a
	\leq \|a\|\,|h|.
	\]
	As $J_y$ is an order ideal containing $|h|$, it follows that
	$h\cdot a\in J_y$. Hence
	\[
	[g\cdot a]_y
	=
	[g(y)a+h\cdot a]_y
	=
	g(y)[a]_y,
	\]
	proving \eqref{eq:fiber-module-action}.
	Consequently, if
	$
	a=\sum_{i=0}^n c_i e_i,
	$
	then
	\begin{equation*}
		[a]_y
		=
		\sum_{i=0}^n c_i(y)[e_i]_y.
	\end{equation*}
	Thus $\{[e_0]_y,\ldots,[e_n]_y\}$ spans $A_y$. Since Axiom~(5)
	gives $\dim A_y=n+1$, these vectors form a basis of $A_y$. In
	particular, the coefficient vector of $[a]_y$ with respect to this
	basis is
	\[
	\bigl(c_0(y),\ldots,c_n(y)\bigr)\in\mathbb R^{n+1}.
	\]
\end{remark}

\begin{remark}
	If $\{e_i\}_{i=0}^n$ and $\{e_i'\}_{i=0}^n$ are two module bases of
$A$ over $\cM_A\cong C(Y)$, the corresponding change-of-basis matrix
$
M(y)\in \operatorname{GL}_{n+1}(\mathbb R)
$
depends continuously on $y\in Y$. Since $Y$ is compact, both
$M(y)$ and $M(y)^{-1}$ are uniformly bounded. It follows that the
topological regularity condition in Axiom~(7) is independent of the
choice of module basis.
\end{remark}

%\begin{proposition}
%	Let $A$ be a free order unit module over $\mathcal{C}(Y)$. The total space
%	\[
%	E = \bigcup_{y \in Y} (A_y \times \{y\})
%	\]
%	equipped with the section topology induced by the bundle convergence in Definition \ref{def:fiber_conv} is Hausdorff.
%\end{proposition}

	\begin{proposition}\label{prop:ns}
	Let $A$ be a free order unit module. Then the following holds:
	\begin{enumerate}
		\item The global norm is consistent with the fiber norms,
		\[ \|a\| = \sup_{y \in Y} \|[a]_y\|_y. \]
		\item The function $y \mapsto \|[a]_y\|_y$ is upper semicontinuous on $Y$.
		\item If $a \in A$ and $y_k\to y_0$ in $Y,$ then 
		\begin{equation*}%\label{eq:vanishing-criterion}
			\|[a]_{y_k}\|_{y_k}\to0
			\quad\Longrightarrow\quad
			[a]_{y_0}=0.
		\end{equation*}
	\end{enumerate}
\end{proposition}

\begin{proof}
	We first prove (\textit{1}).
Since $A$ is an AOU space (over $\R$) with order unit $u,$ the norm on $A$ is defined intrinsically as
$$\|a\| = \inf\{\lambda > 0 \mid -\lambda u \le a \le \lambda u\}.$$ 

By Axiom (5),
\[
-\lambda u \le a \le \lambda u \iff  \quad -\lambda [u]_y \le_{A_y} [a]_y \le_{A_y} \lambda [u]_y \quad \forall y \in Y,
\]
and  $[u]_y = u_y.$ Thus
\[
-\lambda u \le a \le \lambda u \iff  \quad \|[a]_y\|_y \le \lambda \quad \forall y \in Y.
\]
Taking the infimum over $\lambda$, we obtain $\|a\| = \sup_{y \in Y} \|[a]_y\|_y$.

To prove upper semicontinuity (USC) in (\textit{2}),
we show that for any $\lambda \geq 0$, the set $L_{\lambda} = \{y \in Y \mid N(y) < \lambda\}$ is open. So if $y_0 \in Y$ satisfies $N(y_0) < \lambda$, then $N(y) < \lambda$ must hold for all $y$ in a neighborhood of $y_0$.
The condition $N(y_0) < \lambda$ implies that $\lambda [u]_{y_0} \pm [a]_{y_0}$ are strictly positive elements, i.e., interior points, of the cone $A_{y_0}^+$. Let $v_0 = \lambda [u]_{y_0} - [a]_{y_0}$. Since $A_{y_0}^+$ has nonempty interior by Lemma \ref{le:unit}, there exist linearly independent positive elements $[z_1]_{y_0}, \dots, [z_{n+1}]_{y_0} \in A_{y_0}^+$ such that $v_0$ lies in the interior of their conic hull (i.e., the set of all linear combinations with nonnegative coefficients of these elements).

By Remark \ref{re:lhc1}, the map $y \mapsto A_y^+$ is LHC. Thus, for each $i$ there exists a $z_i \in A$ such that $[z_i]_y \in A_y^+$ and $[z_i]_y \to [z_i]_{y_0}$ as $y \to y_0$.
Now the element $v= \lambda u - a \in A$ satisfies $[v]_{y_0} = v_0.$ The latter is in the interior of the conic hull of $[z_1]_{y_0}, \dots, [z_{n+1}]_{y_0}$ and we claim that by continuity, $[v]_y$ lies in the interior of the conic hull of  $[z_1]_y, \dots, [z_{n+1}]_y$ for all $y$ sufficiently close to $y_0$. Indeed, to find an expression of $[v]_y$ as a conic combination (i.e., linear combination with nonnegative coefficients)
\[
[v]_y = \sum_{i=1}^{n+1} c_i(y) [z_i]_y
\]
we need to solve the linear system
\[
\mathbf{v}(y) = Z(y) \cdot \mathbf{c}(y),
\]
where $\mathbf{c}(y) = (c_1(y), \dots, c_{n+1}(y))^T$ is the vector of unknown coefficients, $Z(y)$ is the square matrix whose columns are the coordinate vectors of $[z_1]_y, \dots, [z_{n+1}]_y$ with respect to the standard basis $[e_j]_y$ of $A_y$ and $\mathbf{v}(y)$ is the coordinate vector of $	[v]_y$.

The matrix $Z(y_0)$ is invertible as the $[z_i]_{y_0}$ are linearly independent, meaning $\det(Z(y_0)) \neq 0$. Since the determinant is a continuous function of the entries and the entries of $Z(y)$ vary continuously with $y$ (since $[z_i]_y \to [z_i]_{y_0}$ as $y \to y_0$), there exists a neighborhood of $y_0$ where $\det(Z(y)) \neq 0$.
So for $y$ in this neighborhood, we can solve for the coefficients uniquely,
\[
\mathbf{c}(y) = Z(y)^{-1} \mathbf{v}(y).
\]
We know that $c_i(y_0) > 0$ for all $i=1, \dots, n+1$ as $v_0$ lies in the interior of the conic hull of the $[z_i]_{y_0}.$ Since the functions $y \mapsto c_i(y)$ are continuous, we have $c_i(y)>0$ in a neighbourhood of $y_0.$

The elements in the conic hull of positive elements are positive, thus $[v]_y \in A_y^+.$
Applying the same argument to $\lambda [u]_y + [a]_y$, we conclude that there is a neighborhood $U$ of $y_0$ such that for all $y \in U$, $\lambda [u]_y \pm [a]_y \in A_y^+$. This implies $N(y) \le \lambda$ (and in fact $N(y) < \lambda$ if we apply the argument to any $\lambda' < \lambda$), proving USC.

To prove (\textit{3}) suppose, towards a contradiction, that $[a]_{y_0}\neq0$. Write
\[
a=\sum_{i=0}^n c_i e_i,
\qquad c_i\in \cC(Y).
\]
By Remark~\ref{re:conv}, there is an index $j$ such that
$
c_j(y_0)\neq0.
$
Hence, after discarding finitely many terms,
\begin{equation}\label{eq:nonzero-coefficient}
	|c_j(y_k)|
	\geq \frac12|c_j(y_ 0)|.
\end{equation}

Choose a sequence $\lambda_k>0$ such that
\[
\|[ a]_{y_k}\|_{y_k}<\lambda_k,
\qquad
\lambda_k\to0.
\]
By the upper semicontinuity established in Item (\textit{2}), for each $k$ there is
an open neighborhood $U_k$ of $y_k$ such that
\[
\|[ a]_y\|_y\leq2\lambda_k,
\qquad y\in U_k.
\]
Since $Y$ is a compact subset of $\mathbb R^m$, it is normal. Thus we
can choose $f_k\in \cC(Y)$ such that
\[
0\leq f_k\leq1,\qquad
f_k(y_k)=1,\qquad
\operatorname{supp}(f_k)\subseteq U_k.
\]
Set
$
b_k=f_k\cdot a.
$
By \eqref{eq:fiber-module-action} and Item (\textit{1}),
\[
\begin{split}
	\|b_k\|
	&=
	\sup_{y\in y}\|[f_k\cdot a]_y\|_y \\
	&=
	\sup_{y\in y}
	|f_k(y)|\,\|[ a]_y\|_y
	\leq2\lambda_k
	\to 0.
\end{split}
\]
Now the coefficient functions of $b_k$ are
$
f_kc_0,\ldots,f_kc_n.
$
By Axiom~(7), the coefficient map
$
\Psi:A\to C(Y)^{n+1}
$
is continuous. Hence there is a constant $C>0$ such that
\[
\max_{0\leq i\leq n}\|f_kc_i\|_\infty
\leq C\|b_k\|.
\]
Consequently,
$
\|f_kc_j\|_\infty\to0.
$
On the other hand,
\[
\|f_kc_j\|_\infty
\geq
|f_k(y_k)c_j(y_k)|
=
|c_j(y_k)|
\geq
\frac12|c_j(y_0)|,
\]
contradicting \eqref{eq:nonzero-coefficient}. This proves (\textit{3}).
\end{proof}

\begin{lemma}\label{le:closed}
	Let $A$ be a free order unit module. Let $E  = \bigcup_{y \in Y} (A_y \times \{y\})$ be the total space of the fibers, equipped with the section topology induced by $A$ (as characterized by Definition \ref{def:fiber_conv}). Then the union of the fiber-wise positive cones,
	$$P = \bigcup_{y \in Y} (A_y^+ \times \{y\})$$
	is a closed subset of $E$. 
%	Let $A$ be a free order unit module. Let $E = \bigcup_{y \in Y} (A_y \times \{y\})$ be the total space of the fibers, equipped with the topology induced by the global trivialization $E \cong \mathbb{R}^{n+1} \times Y$ from the module basis $\{e_0, \dots, e_n\}$. Then the union of the fiber-wise positive cones,
%	\[
%	\mathcal{P} = \bigcup_{y \in Y} \big(A_y^+ \times \{y\}\big),
%	\]
%	is a closed subset of $E$.
\end{lemma}

\begin{proof}
	Let $(a_k, y_k) \in P$ be a sequence converging in $E$ to some $(a_0, y_0) \in E$. By Definition~\ref{def:fiber_conv} this means $y_k \to y_0$ in $Y$ and there exists a global section $a \in A$ such that $[a]_{y_0} = a_0$ and
	\[
	\lim_{k\to\infty} \|a_k - [a]_{y_k}\|_{y_k} = 0.
	\]
	We must show that $(a_0, y_0) \in P$, which is equivalent to proving $a_0 \in A_{y_0}^+$.
	
	First, we show that the fiber norms $\|a_k\|_{y_k}$ are uniformly bounded. For the global element $a \in A$, the definition of the Archimedean order unit norm gives $-\|a\| u \le a \le \|a\| u$. Projecting onto the quotient fiber $A_y$, we have $-\|a\| [u]_y \le [a]_y \le \|a\| [u]_y$, which implies $\|[a]_y\|_y \le \|a\|$ for all $y \in Y$. By the triangle inequality in $A_{y_k}$,
	\[
	\|a_k\|_{y_k} \le \|a_k - [a]_{y_k}\|_{y_k} + \|[a]_{y_k}\|_{y_k} \le \|a_k - [a]_{y_k}\|_{y_k} + \|a\|.
	\]
	Since $\|a_k - [a]_{y_k}\|_{y_k} \to 0$, there exists $r > 0$ such that $\|a_k\|_{y_k} \le r$ for all $k$.
	
	Because $a_k \in A_{y_k}^+$ and $\sup_k \|a_k\|_{y_k} \le r < \infty$, the UHC condition in Axiom~(6) implies that there exists a subsequence $(a_{k_j})$ and an element $w_0 \in A_{y_0}^+$ such that $a_{k_j} \to w_0$ in the fiber bundle. That is, there is an element $b \in A$ with $[b]_{y_0} = w_0$ and $\|a_{k_j} - [b]_{y_{k_j}}\|_{y_{k_j}} \to 0$.
	
	By the triangle inequality in the fibers,
	\begin{align*}
		\|[a - b]_{y_{k_j}}\|_{y_{k_j}} &\le \|[a]_{y_{k_j}} - a_{k_j}\|_{y_{k_j}} + \|a_{k_j} - [b]_{y_{k_j}}\|_{y_{k_j}} \xrightarrow{j\to\infty} 0.
	\end{align*}
	By Item (\textit{3}) of Proposition \ref{prop:ns}, we have $[a - b]_{y_0} = 0$, so $a_0 = [a]_{y_0} = [b]_{y_0} \in A_{y_0}^+$. Thus $(a_0, y_0) \in P$, proving $P$ is closed in $E.$
	%Evaluating at $y_0$ gives $\|a_0 - w_0\|_{y_0} = \|[a - b]_{y_0}\|_{y_0} = 0$, so $a_0 = w_0 \in A_{y_0}^+$. Thus $(a_0, y_0) \in P$, which proves that $P$ is closed in $E$.
\end{proof}

Before the next proposition we recall a useful fact about AOU spaces.
	
\begin{lemma}\label{le:unit}
	If $A$ is an AOU space, then its positive cone $A^+$ has non-empty interior. In fact, the order unit $u$ of $A$ is an interior point of $A^+.$
\end{lemma}

\begin{proof}
	Recall that for any $x \in A,$
	$$
	\|x\| = \inf \{r >0 \ | \ -r u \leq x  \leq r u\}
	$$ is the norm on $A.$
	Let $B = \{x \in A \ | \ \|x-u\|<1\}$ be the open unit ball in $A$ centered at $u.$ If $x \in B,$ then by definition of the norm, there exists $0 \leq r < 1$ such that 
	$$
	-r u \leq x-u \leq r u.
	$$
	Hence $x \geq u (1-r) \geq 0,$ which proves that $B \subseteq A^+$ and $u$ is an interior point. 
\end{proof}

We are now in a position to prove an addendum to Proposition~\ref{prop:ns}, showing that the fiber norm function is in fact continuous.

	\begin{proposition}\label{prop:normc}
		Let $A$ be a free order unit module with $\cM_A \cong \cC(Y)$. Then the function $y \mapsto \|[a]_y\|_y$ is continuous on $Y$.
	\end{proposition}
	
	\begin{proof}
		For a fixed $a \in A$ let $N(y) = \|[a]_y\|_y.$
		%Since the fibers are finite-dimensional, it suffices to 
		We prove continuity of $N$ by showing that $N$ is both lower semicontinuous (LSC) and upper semicontinuous (USC). The latter follows from Item (\textit{2}) of Proposition \ref{prop:ns}.
		
		To prove lower semicontinuity
		we show that for any $\lambda \geq 0$, the set $L_{\lambda} = \{y \in Y \mid N(y) \le \lambda\}$ is closed.
		Recall that in an AOU space, $\|v\| \le \lambda$ if and only if $\lambda u \pm v \in A^+$. Thus,
		\[
		L_{\lambda} = \{y \in Y \mid \lambda [u]_y \pm [a]_y \in A_y^+\}.
		\]
		If $(y_k)_k$ is a sequence in $L_{\lambda}$ converging to $y_0,$ then
		\[
		\lim_{k \to \infty} (\lambda [u]_{y_k} \pm [a]_{y_k}) = \lambda [u]_{y_0} \pm [a]_{y_0}
		\]
		in the sense of Definition \ref{def:fiber_conv}.
		%in the total space topology of the bundle $A$.
		By Lemma \ref{le:closed}, the total positive cone $\mathcal{P} = \bigcup_{y \in Y} (A_y^+ \times \{y\})$ is closed. Hence, $\lambda [u]_{y_0} \pm [a]_{y_0} \in A_{y_0}^+$, which implies $N(y_0) \le \lambda$. Hence, $L_{\lambda}$ is closed, and $N$ is LSC.
		Since $N$ is both LSC and USC, it is continuous.
%		 By Kadison's theorem \cite{??}, the AOU space $A_y$ is isomorphic to the space of affine functions on its state space $S_y.$ The latter can be identified with a compact convex set $K_y \subseteq \mathbb{R}^n.$ So an element $[a]_y$ corresponds to an affine function $z \mapsto f_a(z, y)$ on $K_y$, where
%		\[
%		f_a(z, y) = c_0(y) + \sum_{i=1}^n c_i(y)z_i, \quad \text{for } z = (z_1, \dots, z_n) \in K_y.
%		\]
%		We thus have
%		\[
%		N(y) = \|[a]_y\|_{A_y} = \sup_{z \in K_y} |f_a(z, y)|.
%		\]
%		
%		By Axiom $7,$  the functions $c_i : Y \to \mathbb{R}$ are continuous. Hence, the function $F(z, y) = |c_0(y) + \sum c_i(y)z_i|$ is jointly continuous on $\mathbb{R}^n \times Y$.
%		
%
%		 Next, by Proposition \ref{prop:Sreg}, the map $y \to K_y$ is LHC. By Lemma \ref{le:Kclosed}, the set $K = \bigcup_y K_y$ is closed. Since each $K_y$ is compact, the map $y \to K_y$ is also UHC and thus continuous.
%		 Now Berge's Maximum Theorem implies that the maximum value function
%		\[
%		y \mapsto \max_{z \in K_y} F(z, y) = \|[a]_y\|_{A_y}
%		\]
%		is continuous.
	\end{proof}

\begin{proposition}\label{prop:conv}
	Let $y_k \to y_0$ in $Y,$ let $([a_k]_{y_k})_k$ be a sequence with 
	$$[a_k]_{y_k} = \sum_{i=0}^n c^k_i [e_i]_{y_k} \in A_{y_k}$$  and let $[a_0]_{y_0} = \sum_{i=0}^n c_i [e_i]_{y_0}  \in A_{y_0}$. For any $k \geq 0$ denote by $\mathbf{c}^{(k)} \in \mathbb{R}^{n+1}$ the coefficient vector of $[a_k]_{y_k}.$ 
	
	Then the sequence $([a_k]_{y_k})_k$ converges to $[a_0]_{y_0}$ (in the sense of Definition \ref{def:fiber_conv}) if and only if the coefficient vectors converge, $\mathbf{c}^{(k)} \to \mathbf{c}^{(0)}$ in $\mathbb{R}^{n+1}$.
\end{proposition}

\begin{proof} To ease the notation we write $[a_k]_{y_k}= a_k.$
To see	($\Leftarrow$),
	assume $c_i^{(k)} \to c_i^{(0)}$ for each $i$.
	We want to show that $\|a_k - [a]_{y_k}\|_{y_k} \to 0$ for an element $a \in A$ satisfying $[a]_{y_0}=a_0.$
	Let
	\[ a = \sum_{i=0}^n c_i^{(0)} e_i. \]
	Then $a \in A,$ we have $[a]_{y_0}=a_0$ and
	\begin{align*}
		\| a_k -  [a]_{y_k} \|_{y_k} &= \left\| \sum_{i=0}^n c_i^{(k)} [e_i]_{y_k} - \sum_{i=0}^n c_i^{(0)} [e_i]_{y_k} \right\|_{y_k} \\
		&= \left\| \sum_{i=0}^n (c_i^{(k)} - c_i^{(0)}) [e_i]_{y_k} \right\|_{y_k} \\
		&\leq \sum_{i=0}^n |c_i^{(k)} - c_i^{(0)}| \cdot \| [e_i]_{y_k} \|_{y_k}.
	\end{align*}
	By Proposition \ref{prop:normc}, the function $y \mapsto \|[e_i]_y\|_y$ is continuous on the compact set $Y$, so it is bounded by some constant $M$. Hence
	\[ \| a_k -  [a]_{y_k} \|_{y_k} \leq M \sum_{i=0}^n |c_i^{(k)} - c_i^{(0)}|. \]
	Since $\mathbf{c}^{(k)} \to \mathbf{c}^{(0)}$, the above implies $a_k \to a_0$.
	
	For ($\Rightarrow$),
	assume $a_k \to a_0$. By definition, there exists an element $a \in A$ such that $[a]_{y_0} = a_0$ and $\|a_k -  [a]_{y_k}\|_{y_k} \to 0$.
	Let $\mathbf{d}(y) = (d_0(y), \dots, d_n(y))$ be the vector of coefficient functions pertaining to $a$. By Axiom (7), the functions $d_i: Y \to \mathbb{R}$ are continuous, so $y_k \to y_0$ implies $d_i(y_k) \to d_i(y_0) = c_i^{(0)}$.
	
Write
	\[ a_k -  [a]_{y_k} = \sum_{i=0}^n (c_i^{(k)} - d_i(y_k)) [e_i]_{y_k} \]
	and let $\delta^{(k)} \in \mathbb{R}^{n+1}$ be the vector with $\delta_i^{(k)} = c_i^{(k)} - d_i(y_k)$. By assumption,
	 $$\Big\|\sum_i \delta_i^{(k)} [e_i]_{y_k}\Big\|_{y_k} \to 0$$   and we need to show that $\delta_i^{(k)} \to 0.$

	Consider the function $F: Y \times \mathcal{S}^n \to \mathbb{R}$ defined on the product of $Y$ and the unit sphere $\mathcal{S}^n \subset \mathbb{R}^{n+1}$ by
	\[
	F(y, \mathbf{u}) = \bigg\| \sum_{i=0}^n u_i [e_i]_y \bigg\|_y.
	\]
	%For a fixed $\mathbf{u} \in \mathbb{S}^n$, we have $a_{\mathbf{u}} = \sum u_i e_i \in A.$  By Proposition \ref{prop:normc}, the map $y \mapsto \|[a_{\mathbf{u}}]_y\|_{A_y}$ is continuous. Since the norm is also continuous in $\mathbf{u}$, 
	We claim that $F$ is jointly continuous on $Y \times \mathcal{S}^n$. 
In fact, consider
	\[
	|F(y, \mathbf{u}) - F(y_0, \mathbf{u}_0)| \leq |F(y, \mathbf{u}) - F(y, \mathbf{u}_0)| + |F(y, \mathbf{u}_0) - F(y_0, \mathbf{u}_0)|.
	\]
	For a fixed $\mathbf{u}_0 \in \mathcal{S}^n$, let $a = \sum (u_0)_i e_i$. By Proposition \ref{prop:normc}, the map $y \mapsto \|[a]_y\|_y$ is continuous. Thus, as $y \to y_0$, the second term vanishes.
	Next, note that
		\[
		|F(y, \mathbf{u}) - F(y, \mathbf{u}_0)| \leq \left\| \sum (u_i - (u_0)_i) [e_i]_y \right\|_y \leq \sum |u_i - (u_0)_i| \cdot \|[e_i]_y\|_y
		\]
		Since $Y$ is compact and the map $y \mapsto \|[e_i]_y\|$ is continuous by Proposition \ref{prop:normc}, the norms of the basis vectors are bounded globally. Let $M = \max_{i} \sup_{y \in Y} \|[e_i]_y\| < \infty$. Then
		\[
		|F(y, \mathbf{u}) - F(y, \mathbf{u}_0)| \leq M \cdot \|\mathbf{u} - \mathbf{u}_0\|_1.
		\]
	Since $F$ is continuous in $y$ and uniformly continuous in $\mathbf{u}$, it is jointly continuous. 
	
	Now, $\{[e_i]_y\}_i$ is a basis for every $y$, so the vector $\sum u_i [e_i]_y$ is  nonzero for $\mathbf{u} \in \mathcal{S}^n$. So $F(y, \mathbf{u}) > 0$ on the compact space $Y \times \mathcal{S}^n,$ thus  $F$ attains a strictly positive minimum $m > 0$.
	Therefore, for any $y \in Y$ and $\mathbf{v} \in \mathbb{R}^{n+1} \setminus \{0\}$,
	\[
	\left\| \sum_{i=0}^n v_i [e_i]_y \right\|_y 
	= \|\mathbf{v}\| \cdot F\left(y, \frac{\mathbf{v}}{\|\mathbf{v}\|}\right) 
	\geq m \|\mathbf{v}\|.
	\]
	Applying this to $\mathbf{v}=\delta^{(k)}$ yields
	\[
	\|\delta^{(k)}\| \leq \frac{1}{m} \left\| \sum_{i=0}^n \delta_i^{(k)} [e_i]_{y_k} \right\|_{y_k}.
	\]
	Since the term on the right converges to 0, we conclude that $\|\delta^{(k)}\| \to 0$.
\end{proof}

We have now acquired the tools to verify that our motivating example fits the abstract framework. That is, we show that the space $\Caff$
associated with a compact regular partially convex set $K$
satisfies the axioms of a free order unit module.

\begin{proposition}\label{prop:axiomok}
		If $K \subseteq \R^{n+m}$ is a compact regular partially convex set in $x$, then $\Caff$ is a free order unit module of type $(n,m)$ over $\cC(Y),$ where $Y = \pi_2(K).$ 
\end{proposition}

\begin{proof}
	Let $A = \mathcal{C}_{\mathrm{aff}}(K)$ and let $\mathcal{M}_A = \mathcal{C}(Y)$.
	
	 \textit{AOU space structure:} 
		It is straightforward to check that $A$ is an Archimedean order unit space when equipped with the standard pointwise ordering and the constant function $e(x,y) = 1$ as the Archimedean order unit.

		\textit{Module Structure:} 
			%The multiplicative domain $\cM_A$ of $A$ is precisely the space $\cC(Y)$ embedded in $\cC(K)$ as functions constant in $x$. For any $g \in \mathcal{C}(Y)$ and $f \in A$, we have $fg \in A,$ making $A$ a module over $\mathcal{C}(Y)$.
			The multiplicative domain $\mathcal{M}_A$ consists of all $g \in A = C_{\mathrm{aff}}(K)$ such that $g \cdot f \in A$ for all $f \in A$. Since $\operatorname{int}(K_y) \neq \emptyset$ for all $y \in Y$, a function whose product with every coordinate function $x_i$ remains affine on $K_y$ must be constant in $x$ on each slice. Thus, $\mathcal{M}_A \cong C(Y)$ embedded in $C(K)$ as functions depending only on $y$. Furthermore, the $m$ coordinate projection functions $y \mapsto y_j$ ($j = 1, \dots, m$) separate points on the compact set $Y \subseteq \mathbb{R}^m$, so by the Stone--Weierstrass theorem they generate $\mathcal{M}_A \cong C(Y)$ as a unital commutative $C^*$-algebra on $m$ generators.
		%	Let $g \in C(Y)$ and $f \in \Caff(K)$. Since $f$ is affine in $x$ for fixed $y$, the product $g(y)f(x,y)$ remains affine in $x$ (it is simply a scalar multiplication of the affine function on the slice $K_y$). Since the product of continuous functions is continuous, $g \cdot f \in \Caff(K)$.
		%Thus, $C(Y) \subseteq \cM_A$, making $A$ a module over $C(Y)$.
%		For any $g \in \mathcal{C}(Y)$ and $f \in A$, define the product $(g \cdot f)(x, y) = g(y)f(x, y)$. Since $f$ is affine in $x$ for fixed $y$, and $g(y)$ is constant with respect to $x$, the product $g(y)f(x, y)$ remains affine in $x$. Since both functions are continuous, their product is continuous. Thus, $g \cdot f \in \mathcal{C}_{\mathrm{aff}}(K)$, making $A$ a module over $\mathcal{C}(Y)$.
		
	\textit{Finitely Generated and Free:} 
	 Since $K$ is assumed regular, Theorem \ref{thm:ctsXY} asserts that any $f \in \Caff$ admits a unique representation
		\[
		f(x,y) = c_0(y) + \sum_{i=1}^n c_i(y)x_i,
		\]
		where $c_i: Y \to \R$ are continuous functions, i.e., $c_i \in \cM_A$. The basis elements of the free module $A$ are the constant function (and order unit) $e_0 = 1$ and the coordinate functions $e_i = x_i$ for $i=1,\dots,n$.
		Thus, $A$ is a free module of rank $n+1$.
		
	 \textit{Unit Alignment:} 
		The order unit $u$ of $A$ is the constant function $1 = 1_K,$ and the unit of the $C^*$-algebra $\mathcal{C}(Y)$ is the constant function $1_Y$. The first basis vector is $e_0 = 1$. Therefore, $u = 1_Y \cdot e_0$, satisfying the alignment axiom.
		
		 \textit{Fiber-wise Order:} 
			We claim that for any $y \in Y$, the fiber space $A_y = A / J_y$ is isomorphic to $\mathrm{Aff}(K_y)$, the space of continuous affine functions on the slice $K_y$.
		To see this, fix $y_0 \in Y$ and let $\rho_{y_0}: A \to \mathrm{Aff}(K_{y_0})$ be the evaluation map defined by restricting a function $f \in \Caff$ to the slice at $y_0,$
			\[
			\rho_{y_0}(f)(x) = f(x, y_0).
			\]
			This map is clearly well-defined and surjective.
%			\textbf{Surjectivity:} Let $h \in \mathrm{Aff}(K_{y_0})$. Since $h$ is an affine function on a subset of $\mathbb{R}^n$, it can be written as $h(x) = \alpha_0 + \sum_{i=1}^n \alpha_i x_i$ for scalars $\alpha_i \in \mathbb{R}$. We can lift this to a global function $F \in \mathcal{C}_{\mathrm{aff}}(K)$ by setting the coefficient functions to be constants: $c_i(y) = \alpha_i$ for all $y \in Y$. By Theorem 3.5, $F(x,y) = \sum \alpha_i e_i(x,y)$ is in $\mathcal{C}_{\mathrm{aff}}(K)$. Clearly, $\rho_{y_0}(F) = h$.
				 The kernel of the evaluation map is
				\[
				\ker(\rho_{y_0}) = \{ f \in A \mid f(x, y_0) = 0\  \forall x \in K_{y_0} \}.
				\]
				By Theorem \ref{thm:ctsXY}, any $f \in A$ is uniquely represented as $f(x, y) = \sum_{i=0}^n c_i(y)x_i.$ If $f \in \ker(\rho_{y_0})$, then $\sum_{i=0}^n c_i(y_0)x_i = 0$ for all $x \in K_{y_0}$.
				
				Since $K$ is regular, we have $\mathrm{int}(K_{y_0}) \neq \emptyset,$ which implies that the coordinate functions $\{1, x_1, \dots, x_n\}$ are linearly independent on $K_{y_0}$. So $f \in \ker(\rho_{y_0})$  if and only if  $c_i(y_0) = 0$ for all $i$.
				This implies that $c_i \in I_{y_0} = \{ g \in \mathcal{C}(Y) \mid g(y_0) = 0 \}$. Since $J_{y_0}$ is closed order ideal generated by $I_{y_0}$, we have $f \in J_{y_0}$. On the other hand, we clearly have $J_{y_0}\subseteq \ker(\rho_{y_0}).$
				Thus $\ker(\rho_{y_0}) = J_{y_0}$ and we deduce that the induced map $\tilde{\rho}_{y_0}: A_{y_0} = A / J_{y_0} \to \mathrm{Aff}(K_{y_0})$ is a linear bijection.
				
				We now prove that $A_{y_0}$ and $\mathrm{Aff}(K_{y_0})$ are isomorphic as AOU spaces,
				meaning that $\tilde{\rho}_{y_0}$ preserves the order unit and positivity in both directions.
				
				 The order unit $u$ of $A$ satisfies
				\[
				\tilde{\rho}_{y_0}([u]_{y_0})(x) = u(x, y_0) = 1, \quad \forall x \in K_{y_0}.
				\]
				The latter is precisely the order unit of $\operatorname{Aff}(K_{y_0})$.

				It remains to show that 
				$$[f]_{y_0} \in (A_{y_0})^+ \iff \tilde{\rho}_{y_0}([f]_{y_0}) \ge 0 \text{ on } K_{y_0}.$$ 
				
				To see $(\Rightarrow),$ 
				let $[f]_{y_0} \in (A_{y_0})^+$. By the definition of the quotient order, there exists a representative $g \in A$ such that $[g]_{y_0} = [f]_{y_0}$ and $g \in A^+$. Since $g \in A^+$, $g(x, y) \ge 0$ for all $(x, y) \in K$. Specifically, restricting to the slice $y_0$, we have $g(x, y_0) \ge 0$ for all $x \in K_{y_0}$. Since $\tilde{\rho}_{y_0}([f]_{y_0})$ is given by evaluating any representative at $y_0$, the image is a non-negative function.
				
				For $(\Leftarrow)$ 
				let $h \in \operatorname{Aff}(K_{y_0})$ be such that $h(x) \ge 0$ for all $x \in K_{y_0}$. We must show that there exists a global function $G \in A^+$ such that $G(\cdot, y_0) = h$. 
				%We separate two cases.
				
				%\textit{Case 1:} 
				%Assume $h(x) > 0$ for all $x \in K_{y_0}$.
				By regularity of $K,$  we can uniquely represent $h$ as $h(x) = c_0 + \sum_{i=1}^n c_i x_i$.
				Let $F \in A$ be the trivial lift of $h$ obtained by extending the coefficients constantly over $Y,$
				\[
				F(x, y) = c_0 + \sum_{i=1}^n c_i x_i.
				\]
				Then $F(\cdot, y_0) = h$. Now let $m: Y \to \mathbb{R}$ be the minimum value function defined by
				\[
				m(y) = \min_{x \in K_y} F(x, y).
				\]
				Since $K$ is regular, the map $y \mapsto K_y$ is LHC. Since $K$ is compact, this map is also upper hemicontinuous (if a set-valued map has closed graph and the codomain is compact, the map is UHC \cite[Theorem 17.11]{AB06}). Since the function $F$ is continuous, Berge's maximum theorem \cite[Theorem 17.31]{AB06} implies that the value function $m(y)$ is continuous on $Y$.
				
				Note that if $h > 0$ on the compact set $K_{y_0}$, then $m(y_0) > 0.$ On the other hand, if $h \geq 0$ and $h$ has a zero, then $m(y_0)=0.$  Define a correction term $\delta \in \mathcal{C}(Y)$ by
				\[
				\delta(y) = \max(0, -m(y)).
				\]
				Note that $\delta(y_0) = 0$ since $m(y_0) \geq 0.$ Thus, $\delta$ belongs to $I_{y_0}.$
				Now define $G \in A$ by
				\[
				G(x, y) = F(x, y) + \delta(y).
				\]
				For any $(x, y) \in K$, we have $G(x, y) \ge m(y) + \max(0, -m(y)) \ge 0$. Thus $G \in A^+.$ Moreover,
				\[
				G(x, y_0) = F(x, y_0) + \delta(y_0) = h(x) + 0 = h(x).
				\]
%				\textit{Case 2:} Assume $h(x) \ge 0$.
%				For any $\epsilon > 0$, let $h_\epsilon = h + \epsilon$. Then $h_\epsilon > 0$ on $K_{y_0}$. By Case 1, $h_\epsilon$ lifts to a positive class $[G_\epsilon]_{y_0} \in (A_{y_0})^+$.
%				Since AOU spaces have closed positive cones, and $\lim_{\epsilon \to 0} h_\epsilon = h$ in the norm topology, the class corresponding to $h$ must also be in $(A_{y_0})^+$.
			We deduce that $\tilde{\rho}_{y_0}$ is an order isomorphism between $A_{y_0}$ and $\mathrm{Aff}(K_{y_0}).$

		 An element $f \in A$ is positive ($f \ge 0$) if and only if $f(z) \ge 0$ for all $z \in K$. This is equivalent to verifying that for every $y \in Y$, the restriction $f(\cdot, y)$ is non-negative on $K_y$. This implies $A^+ = \{ f \in A \mid \forall y \in Y, [f]_y \in (A_y)^+ \}$.

		\textit{Continuity Condition:} 
		We must show that the map $y \mapsto A_y^+ \cap \{a \mid \|a\| \le r\}$ is UHC for any $r\geq 0.$ Let $y_k \to y_0$ in $Y$ and let $a_k \in A_{y_k}^+$ be a sequence such that $\|a_k\|_{y_k} \le r$ for some $r > 0$. We need to show that there exists a subsequence of $(a_k)_k$ converging to some $a_0 \in A_{y_0}^+$.
		
		Since $A$ is a free module of rank $n+1$, each $a_k$ has a unique representation
		\[
		a_k = \sum_{i=0}^n c_i^{(k)} [e_i]_{y_k},
		\]
		where $\mathbf{c}^{(k)} = (c_0^{(k)}, \dots, c_n^{(k)}) \in \mathbb{R}^{n+1}$. By Proposition \ref{prop:conv}, the convergence of elements in the fiber bundle is equivalent to the convergence of their coefficient vectors. Thus, it suffices to show that the sequence $\mathbf{c}^{(k)}$ is bounded in $\mathbb{R}^{n+1}$.
		
		Since $K$ is a regular partially convex set, we have $\text{int}(K_{y_0}) \neq \emptyset$ and there are $n+1$ affinely independent points $v_0, \dots, v_n \in \text{int}(K_{y_0})$. Since the set-valued map $y \mapsto K_y$ is LHC, for each $j,$ there exists a sequence $v_j^{(k)} \in K_{y_k}$ such that $v_j^{(k)} \to v_j$ as $k \to \infty$.
		
		For $k \in \mathbb{N}$ let $M_k \in M_{n+1}(\mathbb{R})$ be the matrix with $j$-th row $[1, (v_j^{(k)})^T].$ Hence $M_k \to M(y_0)$ as $k \to \infty,$ where $M(y_0)$ is the matrix determined by the affinely independent points $v_0, \dots, v_n$. Therefore, $\det(M(y_0)) \neq 0.$ By continuity of the determinant, $M_k$ is invertible for sufficiently large $k$, and the sequence of inverses $\|M_k^{-1}\|$ is bounded.
		
		Let $\mathbf{F}^{(k)} \in \mathbb{R}^{n+1}$ be the vector with components $F_j^{(k)} = a_k(v_j^{(k)})$. Since $a_k \in A_{y_k}$ corresponds to an affine function on $K_{y_k}$ and $\|a_k\|_{y_k} \le r$, we have
		\[
		|F_j^{(k)}| = |a_k(v_j^{(k)})| \le \sup_{x \in K_{y_k}} |a_k(x)| = \|a_k\|_{y_k} \le r.
		\]
		Thus, the sequence $\mathbf{F}^{(k)}$ is bounded. The coefficient vectors satisfy the linear system $M_k \mathbf{c}^{(k)} = \mathbf{F}^{(k)}$. Therefore,
		\[
		\|\mathbf{c}^{(k)}\| \le \|M_k^{-1}\| \|\mathbf{F}^{(k)}\| \le C
		\]
		for some constant $C$. So there exists a subsequence (still denoted by indices $k$) such that $\mathbf{c}^{(k)} \to \mathbf{c}^{(0)}$ for some $\mathbf{c}^{(0)} \in \mathbb{R}^{n+1}$.
		
		Let $a_0 = \sum_{i=0}^n c_i^{(0)} [e_i]_{y_0}$. By Proposition \ref{prop:conv}, we have $a_k \to a_0$ in the fiber bundle. It remains to show $a_0 \in A_{y_0}^+$. For any $x \in K_{y_0}$, by the LHC property of $K$, there exists a sequence $x_k \in K_{y_k}$ such that $x_k \to x$. Since $a_k \in A_{y_k}^+$, we have $a_k(x_k) \ge 0$. Passing to the limit, we obtain
		\[
		a_0(x) = \lim_{k \to \infty} a_k(x_k) \ge 0.
		\]
		Since $x \in K_{y_0}$ was arbitrary, $a_0$ is non-negative on $K_{y_0}$, implying $a_0 \in A_{y_0}^+$. Thus, the map is UHC on bounded slices.

%			 \textit{Topological Regularity:} 
%		We show that the algebraic isomorphism $\Psi: f \mapsto (c_0, \dots, c_n)$ is a homeomorphism between $A$ and $\mathcal{C}(Y)^{n+1}$.
%		
%		Indeed, since $K$ is compact, the coordinate functions $x_i$ are bounded. If the coefficients $c_i(y)$ converge uniformly to $0$, then $f = \sum c_i e_i$ converges uniformly to $0$.
%		
%		On the other hand, as in the proof of Theorem \ref{thm:ctsXY}, the coefficients of $f \in A$ are recovered via $\mathbf{c}(y) = M(y)^{-1}\mathbf{F}(y)$, where $M(y)$ is a matrix of basis evaluations at affinely independent points in $K_y$. Since $\mathrm{int}(K_y) \neq \emptyset$, we have $\det(M(y)) \neq 0$. By the compactness of $Y$ and the continuity of the determinant, $\|M(y)^{-1}\|$ is uniformly bounded. Therefore, if $f$ converges uniformly to $0,$ the coefficient vector $\mathbf{c}$ converges uniformly to $0.$
		
		\textit{Topological Regularity:} We show that the linear isomorphism $\Psi : f \mapsto (c_0, \dots, c_n)$ is a homeomorphism between $A$ and $C(Y)^{n+1}$.
		Indeed, since $K$ is compact, the coordinate functions $x_i$ are bounded on $K$. Thus, if $\sup_{y \in Y} |c_i^{(k)}(y) - c_i(y)| \to 0$ for all $i$, then $f^{(k)} = \sum_{i=0}^n c_i^{(k)} e_i$ converges uniformly to $f$ in $A$.
		Conversely, by Lemma~\ref{lem:coeff_recovery}, there exists a constant $C > 0$ such that $\max_{0 \le i \le n} \|c_i\|_{C(Y)} \le C \|f\|_\infty$. Therefore, if $f^{(k)} \to 0$ uniformly in $A$, each coefficient function $c_i^{(k)}$ converges uniformly to $0$ in $C(Y)$.  
\end{proof}

\section{Fiber-Preserving States}\label{sec:state}

In this section, we define the partially convex state space associated with a free order unit module and prove that it provides the canonical example of a compact regular partially convex set.

\begin{definition} Let $A$ be a free order unit module with multiplicative domain $\cM_A\cong \cC(Y).$
	For a given point $y_0 \in Y$, a \textbf{$y_0$-state} on $A$ is a state $\Phi: A \to \R$ that is $y_0$-linear in the sense that for any $g \in \cC(Y)$ and any $f \in A$,
	\[ \Phi(g \cdot f) = g(y_0) \Phi(f). \]
	
\end{definition}
The last condition ($y_0$-linearity) forces the state $\Phi$ to be ``pinned" to the point $y_0$ in the base space $Y$.

\begin{definition} \label{def:pstate} Let $A$ be a free order unit module with multiplicative domain $\cM_A\cong \cC(Y).$
	For each $y \in Y$, let $S_y(A)$ be the set of all $y$-states on $A$. The \textbf{partially convex state space} of $A$ is 
	the union of the fibers $S_y(A)$ for all $y \in Y,$  
	\[ S_{\mathrm{par}}(A) = \bigcup_{y \in Y} (S_y(A) \times \{y\}) \subseteq A^* \times Y. \]
\end{definition}

The next lemma gives an identification of the partially convex state space $S_{\mathrm{par}}(A)$ with a  bundle of the state spaces of the fibers $A_y$.

\begin{lemma}\label{le:Sy} Let $A$ be a free order unit module.
	Then the set $S_y(A)$ of $y$-states on $A$ is canonically affinely homeomorphic to the state space $S(A_y)$ of the function system $A_y$. 
\end{lemma}

\begin{proof}
	Recall from Definition \ref{def:module} that $A_y = A / J_y$, where $J_y$ is the closed order ideal generated by elements $g \in \cM_A \cong \cC(Y)$ vanishing at $y$.
	
First, let $\Phi \in S_y(A)$. By the definition of a $y$-state, $\Phi$ is $y$-linear: $\Phi(g \cdot a) = g(y)\Phi(a)$. Note that any element of the form $g \cdot a$ with $g(y)=0$ and $a \in A$ satisfies
	\[
	\Phi(g \cdot a) = 0 \cdot \Phi(a) = 0.
	\]
	Since elements of $J_y$ are limits of elements of the form $\sum_{k=1}^m g_k a_k $ with $g_k(y)=0$ and $a_k\in A$ and since $\Phi$ is continuous, 
	we have that $\Phi(f)=0$ for all $f \in J_y.$
	Since $\Phi$ vanishes on $J_y$, it descends to a well-defined functional $\tilde{\Phi}$ on the quotient $A_y$, defined by $\tilde{\Phi}([a]_y) = \Phi(a)$. Since $\Phi$ is a positive unital map, $\tilde{\Phi}$ is a state on $A_y$.
	
	Conversely, let $\psi \in S(A_y)$ be a state on the fiber $A_y.$ Define $\Phi : A \to \mathbb{R}$ by $\Phi(a) = \psi([a]_y)$. By Remark \ref{re:conv}, the module action on the fiber satisfies $[g \cdot a]_y = g(y)[a]_y.$ 
	%(since $g \cdot a - g(y)a = (g-g(y)\cdot 1)a \in J_y$). 
	Therefore,
	\[
	\Phi(g \cdot a) = \psi([g \cdot a]_y) = \psi(g(y)[a]_y) = g(y)\psi([a]_y) = g(y)\Phi(a).
	\]
	Thus, $\Phi$ satisfies the $y$-linearity condition and defines a valid $y$-state on $A$.
\end{proof}

\begin{remark}\label{rem:SK}
	For each $y \in Y,$ the state space $S_y = S_y(A)$ can be equipped with coordinates. Indeed, for $i \in \{1, \dots, n\}$, let $x_i: S_y \to \R$ be the coordinate function $x_i(\psi) = \psi([e_i]_y)$.
	Then $\iota_y: S_y \to \R^n,$
	\[ \iota_y(\psi) = (x_1(\psi), \dots, x_n(\psi)) = (\psi([e_1]_y), \dots, \psi([e_n]_y)), \]
	is continuous and affine, and it is injective since states are uniquely determined by their values on $\{[e_1]_y, \dots, [e_n]_y\}$ (recall that $\psi([e_0]_y) = 1$ by definition). Let $K_y = \iota_y(S_y) \subset \R^n$. Since $S_y$ is a compact convex set and $\iota_y$ is an affine homeomorphism onto its image, $K_y$ is also a compact convex set.\looseness=-1
\end{remark}

\begin{definition}\label{def:state-conv} Let $A$ be a free order unit module with multiplicative domain $\cM_A \cong \cC(Y).$ For $y \in Y,$ let $K_y$ be as in Remark \ref{rem:SK}.
	Denote by $K = K_A$ the union of the fibers $K_y$ in the product space,
	\[ K = \bigcup_{y \in Y} (K_y \times \{y\}) \subseteq \R^n \times Y \subseteq \R^{n+m}. \]
	The set $K$ is called the \textbf{partially convex coordinate state space} of $A.$
	%We say that the partially convex state space $S_{\mathrm{par}}(A)$ is \textbf{regular} if the corresponding set $K$ is regular (as per Definition \ref{def:reg}).
\end{definition}

\begin{remark}\label{re:choice}
	Definition \ref{def:state-conv} uses a choice of module basis $\{e_{0},\dots,e_{n}\}$ for the free order unit module $A$ and a choice of generating tuple $(g_1, \dots, g_m)$ for $\mathcal{M}_A$ to represent the partially convex state space as a coordinate subset $K_{A}\subseteq\mathbb{R}^{n+m}$. While changing the generating tuple alters the Euclidean realization $Y \subseteq \mathbb{R}^m$ of the base (without changing its intrinsic character space $\Omega(\mathcal{M}_A)$) and changing the module basis alters the Euclidean slices $K_y \subset \mathbb{R}^n$, the resulting coordinate sets are isomorphic in the sense of Definition \ref{def:part-iso} via the induced homeomorphisms $\phi: Y \to Y'$ on the base and affine homeomorphisms $\theta_y: K_y \to K'_{\phi(y)}$ on the slices.
	
	Moreover, $S_{par}(A)$ carries a natural intrinsic topology independent of any basis or generator choice, defined via evaluation: a sequence $\phi_{k}\in S_{y_{k}}$ converges to $\phi\in S_{y}$ if $y_{k}\rightarrow y$ in $\Omega(\mathcal{M}_A)$ and $\phi_{k}([a]_{y_{k}})\rightarrow\phi([a]_{y})$ for every $a\in A$. 
	Since $A$ is a finitely generated free $\cC(Y)$-module, this intrinsic topology coincides with the coordinate topology. 
	
	Indeed, 
	assume $\phi_k \to \phi$ in coordinates. This means $y_k \to y$ and $\phi_k([e_i]_{y_k}) \to \phi([e_i]_{y})$ for all $i=0,\dots,n.$
	Now for any $a \in A$ there exist unique coefficient functions $c_i \in \cC(Y)$ (which are continuous by Axiom (7)) such that $a = \sum_{i=0}^n c_i e_i.$ 
	Then for any $z \in Y$ and $\phi \in A_z^*,$ we have $\phi([a]_{z}) = \sum_{i=0}^n c_i(z) \phi([e_i]_{z}).$
Hence,
	\[
	\lim_{k \to \infty} \phi_k([a]_{y_k}) = \lim_{k \to \infty} \sum_{i=0}^n c_i(y_k) \phi_k([e_i]_{y_k}).
	\]
	Since the $c_i$ are continuous, $c_i(y_k) \to c_i(y),$ and by assumption, $\phi_k([e_i]_{y_k}) \to \phi([e_i]_{y}).$
	Therefore, the sum converges to $\sum_{i=0}^n c_i(y) \phi([e_i]_{y}) = \phi([a]_{y})$.
	On the other hand, it is clear that intrinsic convergence implies convergence in coordinates.
	In what follows we use the coordinate definition of convergence (Definition \ref{def:state-conv}) knowing it creates no topological ambiguity.
\end{remark}

The next lemma and proposition show that if $A$ is a free order unit module of type $(n,m),$ then the corresponding set $K_A \subseteq \R^{n+m}$ is a compact regular partially convex set.

\begin{lemma}\label{le:Kclosed}
	Let $A$ be a free order unit module. Then the partially convex coordinate state space $K = K_A$ of $A$ is a compact set.
\end{lemma}

\begin{proof}
	Let 
	$$(x^{(k)}, y_k) = (x_1(\psi_k), \dots, x_n(\psi_k),y_k) = (\psi_k([e_1]_{y_k}), \dots, \psi_k([e_n]_{y_k}), y_k) \in K_{y_k} \times \{y_k\}$$ 
	be a sequence converging to 
	$(x_1,\ldots,x_n, y) \in \R^{n+m}.$
	Since $Y$ is compact, $y \in Y.$ It remains to prove that there is a $\psi \in S_y$ such that $(x_1,\ldots,x_n) = (x_1(\psi), \dots, x_n(\psi)) \in K_y.$ 
	
	Define $\psi : A_y \to \R$ via
	$$
	\psi([e_i]_y) = \lim_k \psi_k([e_i]_{y_k})
	$$
	for $i=0,\ldots,n.$ 
	Since $\psi_k([e_0]_{y_k})=1$ for all $k,$ we have $\psi([e_0]_y) =1.$
	
	Now let $p =  \sum_{i=0}^n c_i(y) [e_i]_y \in A_y^+.$ Since $y_k \to y,$ Remark \ref{re:lhc1} implies there is a sequence $(p_k)_k$ with $p_k \in A_{y_k}^+$ such that $p_k \to p$ in the sense of Definition \ref{def:fiber_conv}.
	So there exists an $a \in A$ such that $[a]_y = p$ and
	\[
	\lim_{k \to \infty} \| p_k - [a]_{y_k} \|_{y_k} = 0.
	\]
	By Axiom (3), the element $a$ has a unique representation $a = \sum_{j=0}^n \alpha_j e_j$ with $\alpha_j \in \mathcal{M}_A \cong \cC(Y).$ Note that
	\begin{align*}
		\psi_k([a]_{y_k}) &= \psi_k\left( \sum_{j=0}^n \alpha_j(y_k) [e_j]_{y_k} \right) \\
		&= \sum_{j=0}^n \alpha_j(y_k) \psi_k([e_j]_{y_k}).
	\end{align*}
	Since $\alpha_j$ are continuous and $x^{(k)} \to x$, taking the limit gives
	\[
	\lim_{k \to \infty} \psi_k([a]_{y_k}) = \sum_{j=0}^n \alpha_j(y) x_j = \psi\left( \sum_{j=0}^n \alpha_j(y) [e_j]_y \right) = \psi([a]_y) = \psi(p),
	\]
	where we set $x_0=1.$
	Now each $\psi_k$ has norm $1,$ hence
	\[
	| \psi_k(p_k) - \psi_k([a]_{y_k}) | = | \psi_k(p_k - [a]_{y_k}) | \le \| p_k - [a]_{y_k} \|_{y_k}.
	\]
	The right-hand side converges to 0 and since $p_k \in A_{y_k}^+$, we have $\psi_k(p_k) \ge 0$. Thus
	\[
	\psi(p) = \lim_{k \to \infty} \psi_k([a]_{y_k}) = \lim_{k \to \infty} \psi_k(p_k) \ge 0.
	\]
	Since $p$ was arbitrary, $\psi$ is a positive functional. Hence $\psi \in S_y(A)$, which implies $x \in K_y$ and thus $(x,y) \in K$, proving that $K$ is closed. 
	
	To see that $K$ is bounded, note that for any $(x,y) \in K$ with $x = (\psi([e_1]_y), \dots, \psi([e_n]_y))$ for some $\psi \in S_y(A)$, we have
	\[
	|x_i| = |\psi([e_i]_y)| \le \|\psi\| \cdot \|[e_i]_y\|_y \le \|[e_i]_y\|_y \le \|e_i\| \quad (i = 1, \dots, n)
	\]
	by Item (\textit{1}) of Proposition \ref{prop:ns}. Since $Y \subseteq \mathbb{R}^m$ is compact, it follows that $K \subseteq \prod_{i=1}^n [-\|e_i\|, \|e_i\|] \times Y$ is bounded. Being closed and bounded in $\mathbb{R}^{n+m}$, $K$ is compact.
\end{proof}

\begin{proposition}\label{prop:Sreg}
	If $A$ is a free order unit module, then the coordinate state space $K_A$ is a  compact regular partially convex set.
\end{proposition}

\begin{proof} 
	Since each $K_y$ is convex, $K_A$ is a partially convex set. By Lemma \ref{le:Kclosed}, $K_A$ is compact. 
	%Since $K_A$ is homeomorphic to $S_{\mathrm{par}}(A),$ the latter is also compact. 
	It thus remains to prove regularity. 
	
	\textit{Nonempty interior:}
	We first show that for any finite-dimensional real Archimedean order unit space $\cR$ with order unit $u,$ its state space $S(\cR)$ has nonempty interior in the affine hyperplane $\{\varphi \in \cR^* \ | \ \varphi(u) = 1\}$. 
	
	By Lemma \ref{le:unit}, the order unit $u$ of $\cR$ is an interior point of the cone $\cR^+.$  
	Next, since $\cR$ is finite-dimensional and the cone $\cR^+$ is pointed (i.e., $\cR^+ \cap (-\cR^+) = \{0\}$), the dual cone $\cP = \{ \varphi \in \cR^* \ | \  \varphi(\cR^+) \geq 0\}$ has nonempty interior in $\cR^*.$ 
%	Indeed, consider $S = \{ x \in \mathcal{R}^+ \mid \|x\| = 1 \}.$
%	Since $\mathcal{R}$ is finite-dimensional, the unit sphere in $\cR$ is compact. Because $\mathcal{R}^+$ is closed, $S$ is a closed subset of a compact set, and thus $S$ is compact. 
%	
%	Since $\mathcal{R}^+$ is pointed, $0 \notin S$. By the Hahn-Banach separation theorem, there exists a functional $\phi \in \mathcal{R}^*$ and a constant $\epsilon > 0$ such that
%	\[ \phi(x) > \epsilon \quad \text{for all } x \in S. \]
%	By linearity we have $\phi(x) = \|x\| \phi(x/\|x\|) > 0$ for any $x \in \mathcal{R}^+ \setminus \{0\}.$ Hence, $\phi$ is an interior point of $\cP.$\looseness=-1
	Indeed, consider the compact set $S=\{x\in\mathcal{R}^+ \ | \ \Vert{}x\Vert{}=1\}$. Because $\mathcal{R}^+$ is pointed, $0 \notin \operatorname{conv}(S)$. (If $0 = \sum_{j=1}^k \lambda_j s_j$ with $\lambda_j > 0$ and $s_j \in S$, then $-s_1 = \sum_{j=2}^k \frac{\lambda_j}{\lambda_1} s_j \in \mathcal{R}^+$, forcing $s_1 \in \mathcal{R}^+ \cap (-\mathcal{R}^+) = \{0\}$, which contradicts $\Vert{}s_1\Vert{}=1$.) Since $\operatorname{conv}(S)$ is a compact convex set in finite dimensions not containing $0$, the strict Hahn–Banach separation theorem guarantees the existence of a functional $\phi\in\mathcal{R}^*$ and a constant $\varepsilon>0$ such that
	$$\phi(x) \ge \varepsilon > 0 \quad \text{for all } x\in \operatorname{conv}(S),$$
	and in particular $\phi(x) \ge \varepsilon > 0$ for all $x \in S$. By positive homogeneity, $\phi(x) = \Vert{}x\Vert{}\phi(x/\Vert{}x\Vert{}) > 0$ for any $x\in\mathcal{R}^+\setminus\{0\}$. Hence, $\phi$ is an interior point of $\mathcal{P}$. Indeed, any $\psi \in \mathcal{R}^*$ with $\Vert{}\psi - \phi\Vert{} < \epsilon$ satisfies $\psi(x) \ge \phi(x) - \Vert{}\psi - \phi\Vert{} > 0$ for all $x \in S$, which by positive homogeneity ensures $\psi(\mathcal{R}^+) \ge 0$ and thus $B(\phi, \epsilon) \subseteq \mathcal{P}$, proving $\phi \in \operatorname{int}(\mathcal{P}).$

	Note that any nonzero $\varphi \in \cP$ satisfies $\varphi(u) >0.$ Indeed, if $\varphi$ is nonzero, there is an $x \in \cR$ with $\|x\| \leq 1$ such that $\varphi(x) \neq 0.$
	Since $u$ is an interior point of $\cR^+,$ there is an $\epsilon >0$ such that $u \pm \epsilon x \in \cR^+.$ Hence $\varphi(u \pm \epsilon x) \geq 0$ and so $\varphi(u) \geq \epsilon|\varphi(x)| >0.$
	
	It follows that each nonzero $\varphi \in \cP$ can be normalized to a state $\tilde{\varphi} = \varphi(u)^{-1}\varphi \in S(\cR).$ Since $\cP$ has nonempty interior in $\cR^*,$ it is now clear that $S(\cR)$ has nonempty interior in the affine hyperplane $\{ \varphi \in \cR^* \mid \varphi(u) = 1 \}.$ Indeed, for any $\phi \in \operatorname{int}(\mathcal{P})$, the normalized state $\psi = \phi(u)^{-1}\phi$ also lies in $\operatorname{int}(\mathcal{P})$, so an open ball $B(\psi, \delta) \subseteq \mathcal{P}$ yields a relatively open neighborhood $B(\psi, \delta) \cap H \subseteq S(\mathcal{R})$ in the hyperplane $H.$
	
	It follows that the interior of $S_y$ in $\{\varphi \in A_y^* \ | \ \varphi([e_0]_y) = 1\}$ is nonempty for all $y \in Y.$ Hence, the interior of $K_y$ is nonempty for all $y \in Y.$

	\textit{LHC property:}
	It remains to show that the set-valued map $y\mapsto \mathrm{int}(K_{y})$ is LHC. That is, we show that for any sequence $y_{k}\rightarrow y_{0}$ in $Y$ and any $x \in \mathrm{int}(K_{y_{0}})$, there exists a sequence $x_{k}\in \mathrm{int}(K_{y_{k}})$ such that $x_{k}\rightarrow x$.  By Remark \ref{rem:lhc}, it suffices to prove that $y\mapsto K_{y}$ is LHC. 
	
	We proceed by contradiction. Suppose the map is not LHC at $y_{0}$. Then there exists a point $x\in K_{y_{0}}$, a sequence $y_{k}\rightarrow y_{0}$ in $Y$, and $\epsilon>0$ such that for all sufficiently large $k$,
	$$\mathrm{dist}(x,K_{y_{k}})\ge\epsilon.$$ For each $k$, $K_{y_k} \subset \mathbb{R}^n$ is a non-empty compact convex set. Let $p_k = \mathrm{proj}_{K_{y_k}}(x)$ denote the metric projection of $x$ onto $K_{y_k}$. Then $\Vert{}p_k - x\Vert{}_2 \ge \epsilon$. Define the unit vector$$v_k = \frac{p_k - x}{\Vert{}p_k - x\Vert{}_2} \in \mathbb{R}^n \quad (\Vert{}v_k\Vert{}_2 = 1).$$By the standard projection property for closed convex sets, $\langle v_k, z - p_k \rangle \ge 0$ for all $z \in K_{y_k}$. Since $p_k = x + \Vert{}p_k - x\Vert{}_2 v_k$, we obtain$$\langle v_k, z - x \rangle \ge \Vert{}p_k - x\Vert{}_2 \ge \epsilon \quad \text{for all } z \in K_{y_k}.$$ 
	
	Now define the affine function $f_k: \mathbb{R}^n \to \mathbb{R}$ by
	$$f_k(z) = \langle v_k, z - x \rangle - \epsilon = \sum_{i=1}^n (v_k)_i z_i - \langle v_k, x \rangle - \epsilon.$$
	By construction we have $f_k(z) \ge 0$ for all $z \in K_{y_k}$ and $f_k(x) = -\epsilon < 0$.
	Under the identification $\mathrm{Aff}(K_{y_k}) \cong A_{y_k}$, $f_k$ defines an element $a_k \in A_{y_k}^+$ with coefficient vector $c^{(k)} \in \mathbb{R}^{n+1}$ given by$$c_0^{(k)} = -\langle v_k, x \rangle - \epsilon, \quad c_i^{(k)} = (v_k)_i \quad (i = 1, \dots, n).$$

	By Lemma \ref{le:Kclosed}, the coordinate state space $K = K_A$ is compact, so there exists $R > 0$ such that $\Vert{}z\Vert{}_2 \le R$ for all $z \in \bigcup_{y \in Y} K_y$. Since $\Vert{}v_k\Vert{}_2 = 1$ and $x \in K_{y_0}$, the Cauchy–Schwarz inequality yields
	$$\vert{}c_0^{(k)}\vert{} \le \Vert{}v_k\Vert{}_2 \Vert{}x\Vert{}_2 + \epsilon \le R + \epsilon.$$
	Consequently, the sequence of coefficient vectors is uniformly bounded, $$\Vert{}c^{(k)}\Vert{}_2^2 = (c_0^{(k)})^2 + \sum_{i=1}^n (v_k)_i^2 \le (R + \epsilon)^2 + 1 < \infty.$$
	Hence, there exists a subsequence $(c^{(k_j)})_j$ converging to some $c^{(0)} \in \mathbb{R}^{n+1}$. By Proposition \ref{prop:conv}, the corresponding elements $a_{k_j} \in A_{y_{k_j}}^+$ converge in the fiber bundle to $a_0 = \sum_{i=0}^n c_i^{(0)} [e_i]_{y_0} \in A_{y_0}$. Since the total positive cone $\mathcal{P} = \bigcup_{y \in Y} (A_y^+ \times \{y\})$ is closed (Lemma \ref{le:closed}), we have $a_0 \in A_{y_0}^+$.  The element $a_0$ defines a non-negative affine function on $K_{y_0}$. In particular, since $x \in K_{y_0}$, we must have $a_0(x) \ge 0$. On the other hand, evaluating at $x$ yields
	$$a_0(x) = c_0^{(0)} + \sum_{i=1}^n c_i^{(0)} x_i = \lim_{j \to \infty} \left( c_0^{(k_j)} + \sum_{i=1}^n c_i^{(k_j)} x_i \right) = \lim_{j \to \infty} f_{k_j}(x) = -\epsilon < 0,$$
	which is a contradiction. Thus, $y \mapsto K_y$ is LHC, and consequently $y \mapsto \mathrm{int}(K_y)$ is LHC.
\end{proof}

		\section{Duality Theorem} \label{sec:rep}
Having characterized the topological structure of compact regular partially convex sets and the axiomatic properties of free order unit modules, we now formalize their duality. This section is dedicated to proving the analog of Kadison duality for partial convexity, Theorem \ref{thm:dual}. The proof proceeds in two steps: first, we establish a representation theorem for free order unit modules, followed by a representation theorem for compact regular partially convex sets.

	\begin{theorem}\label{th:fnsys}
		%Let $A$ be a free order unit module of type $(n,m).$ Then there exists a compact regular partially convex set $K \subset \R^{n+m}$ such that $A$ is order-isomorphic to $\Caff$.
		Let $A$ be a free order unit module of type $(n, m)$ and let $K_A \subseteq \mathbb{R}^{n+m}$ be its partially convex coordinate state space. Then $K_A$ is a compact regular partially convex set, and $A$ is isomorphic to $C_{\mathrm{aff}}(K_A)$ as a free order unit module.
	\end{theorem}
	
	\begin{proof}

%By Axiom (2), the multiplicative domain $\cM_A$ is a finitely generated commutative $C^*$-algebra on $m$ generators. By the Gelfand-Naimark theorem, there exists a compact Hausdorff space $Y \subseteq \R^m$ (the maximal ideal space of $\cM_A$) such that $\cM_A \cong \cC(Y).$

By Axiom~(2), the multiplicative domain $\mathcal{M}_A$ is a unital commutative real $C^*$-algebra generated by $m$ elements $g_1, \dots, g_m$. By the Gelfand--Naimark theorem, the character space $\Omega(\mathcal{M}_A)$ is compact Hausdorff, and the evaluation map $\chi \mapsto (\chi(g_1), \dots, \chi(g_m))$ embeds $\Omega(\mathcal{M}_A)$ homeomorphically onto a compact subset $Y \subseteq \mathbb{R}^m$, yielding an isometric isomorphism $\mathcal{M}_A \cong C(Y)$.
		
	%{Step 1: Reconstruct the Convex Fibers.}
		For each $y \in Y$, Axiom (5) provides a finite-dimensional AOU space $A_y.$ By the Kadison representation theorem, $A_y$ is order-isomorphic to the space $\Aff(S_y)$ of all continuous affine functions on its state space $S_y$, which is the compact convex set
		\[ S_y = \{ \psi \in A_y^* \mid \psi \text{ is linear, } \psi(A_y^+) \ge 0, \text{ and } \psi([e_0]_y) = 1 \}. \]
		The isomorphism $\Psi_y: A_y \to \Aff(S_y)$ is the evaluation map $\Psi_y(v)(\psi) = \psi(v)$.
		As in Remark \ref{rem:SK}, each $S_y$ is equipped with coordinates to obtain a set $K_y \subseteq \R^n$ and we set $K = \bigcup_y (K_y \times \{y\}).$ 
		By Proposition \ref{prop:Sreg}, the set $K$ is a compact regular partially convex set.

		%{Step 4: Construct the Isomorphism.}
		Define $\Phi: A \to \Caff$ as follows: for an element $a = \sum_{i=0}^n c_i e_i \in A$, the function $\Phi(a)=f_a$ on $K$ acts as 
		\[ f_a(x, y) = c_0(y) + \sum_{i=1}^n c_i(y) x_i, \]
		where $(x, y) \in K$ with $x = (x_1, \dots, x_n) \in K_y.$
		This map is well-defined and linear. It is surjective by Theorem \ref{thm:ctsXY}, and it is injective because the representation of elements of $A$ in the basis is unique.
		
		%{Step 5: Prove it is an Order-Isomorphism.}
		It remains to prove that $\Phi$ is an order-isomorphism. For any $a = \sum c_i e_i \in A$, we have the following chain of equivalences:
		\begin{align*}
			a \in A^+ & \iff [a]_y \in A_y^+ \text{ for all } y \in Y && \text{(by Axiom (5))} \\
			& \iff \text{the function } \Psi_y([a]_y) \text{ is non-negative on } S_y \text{ for all } y \in Y \\
			& \iff \psi([a]_y) \ge 0 \text{ for all } y \in Y \text{ and all } \psi \in S_y \\
			& \iff \psi\left(\sum_{i=0}^n c_i(y) [e_i]_y\right) \ge 0 \text{ for all } y \in Y, \psi \in S_y \\
			& \iff \sum_{i=0}^n c_i(y) \psi([e_i]_y) \ge 0 \text{ for all } y \in Y, \psi \in S_y \\
			& \iff c_0(y) + \sum_{i=1}^n c_i(y) x_i(\psi) \ge 0 \text{ for all } y \in Y, \psi \in S_y \\
			& \iff \Phi(a)(x, y) \ge 0 \text{ for all } (x, y) \in K \quad  && (\text{since } K_y = \iota_y(S_y)) \\
			& \iff \Phi(a) \ge 0 \text{ in } \Caff.
		\end{align*}
		Since $\Phi(u) = \Phi(e_0)$ is the constant function $1$, $\Phi$ is a unital order isomorphism. 
		
		Next, under the identification $\mathcal{M}_A \cong \mathcal{C}(Y),$ the restriction $\rho = \Phi|_{\mathcal{M}_A}$ acts as $\rho(z)(x,y) = z(y)$. By Proposition~\ref{prop:axiomok}, $\mathcal{M}_{\mathcal{C}_{\mathrm{aff}}(K_A)} \cong \mathcal{C}(Y)$, so $\rho: \mathcal{M}_A \to \mathcal{M}_{\mathcal{C}_{\mathrm{aff}}(K_A)}$ is an isometric $*$-isomorphism of $C^*$-algebras.
		
		Finally, for any $z \in \mathcal{M}_A$ and $a = \sum_{i=0}^n c_i e_i \in A$, we have $z \cdot a = \sum_{i=0}^n (z c_i) e_i$. Evaluating at $(x,y) \in K_A$ yields
		\[
		\Phi(z \cdot a)(x,y) = \sum_{i=0}^n (z c_i)(y) x_i = z(y) \sum_{i=0}^n c_i(y) x_i = (\rho(z) \cdot \Phi(a))(x,y)
		\]
		(with $x_0 = 1$). Thus $\Phi(z \cdot a) = \rho(z) \cdot \Phi(a)$, proving that $\Phi$ is an isomorphism of free order unit modules.
	\end{proof}

\begin{theorem} \label{thm:sep}
	Let $K\subseteq\mathbb{R}^{n+m}$ be a compact regular partially convex set in $x\in\mathbb{R}^{n}$, and let $A = \mathcal{C}_{\textnormal{aff}}(K)$ be its associated free order unit module of type $(n,m)$. Then $K$ is isomorphic to the partially convex coordinate state space $K_A \subseteq \mathbb{R}^{n+m}$ in the sense of Definition \ref{def:part-iso}. Moreover, the canonical evaluation map $\theta: K \to S_{par}(A)$ is a fiberwise affine homeomorphism with respect to the intrinsic topology on $S_{par}(A)$.  
\end{theorem}

\begin{proof}
%	Denote $A = \Caff$ and let $S_{\text{par}}(A)$ be the corresponding partially convex state space. Then for $y \in Y,$
%	$$
%	A_y = \{f(\cdot,y) \ | \ f \in \Caff  \} 
%	$$
%	is the space $\Aff(K_y)$ of continuous affine functions on $K_y.$ By Kadison's theorem, the state space $S(A_y)$ of $A_y$ is affinely homeomorphic to $K_y$ via the evaluation map $\text{ev}_x : K_y \to S(A_y)$ sending 
%	$$x\  \mapsto\  \left(\,f(\cdot, y) \mapsto f(x,y)\,\right).$$
%	
%	Since $S(A_y) \cong S_y(A),$ we have that $K_y$ and $S_y(A)$ are affinely homeomorphic for all $y \in Y.$ Hence, the map $\theta: K \to S_{\text{par}}(A),$
%	$$
%	(x,y) \mapsto (\text{ev}_x, y)
%	$$
%	defines an isomorphism of the sets $K$ and  $S_{\text{par}}(A).$
Let $A = \mathcal{C}_{\text{aff}}(K)$ with multiplicative domain $\mathcal{M}_A = \mathcal{C}(Y)$, where $Y = \pi_2(K)$. For each $y \in Y$, the fiber space $$
	A_y = \{f(\cdot,y) \ | \ f \in \Caff  \} 
	$$
	is the space $\Aff(K_y)$ of continuous affine functions on $K_y.$ 
	%By Kadison's theorem, the state space $S(A_y) \cong S_y(A)$ is affinely homeomorphic to the slice $K_y$ via the evaluation map  $$\mathrm{ev}_x: K_y \to S_y(A), \quad \mathrm{ev}_x(f) = f(x,y) \quad \text{for } f \in A_y \cong \mathrm{Aff}(K_y).$$
	By Kadison's theorem, the state space $S(A_y)$ of $A_y$ is affinely homeomorphic to $K_y$ via the evaluation map $\text{ev}_x : K_y \to S(A_y)$ sending 
		$$x\  \mapsto\  \left(\,f(\cdot, y) \mapsto f(x,y)\,\right) \quad \text{ for } f(\cdot, y) \in A_y.$$
		Since $S(A_y) \cong S_y(A)$ by Lemma \ref{le:Sy}, the sets $K_y$ and $S_y(A)$ are affinely homeomorphic for all $y \in Y.$
Define the global evaluation map $\theta: K \to S_{par}(A)$ by 
$$\theta(x,y) = (\mathrm{ev}_x, y).$$  

For each fixed $y \in Y$, the restriction $\theta(\cdot,y): K_y \to S_y(A)$ is an affine homeomorphism. To verify that $\theta$ is a global homeomorphism onto $S_{par}(A)$ equipped with the intrinsic topology of Remark \ref{re:choice}, let $(x_k, y_k) \to (x, y)$ in $K$. 
For any $f \in A = \mathcal{C}_{\text{aff}}(K)$, continuity of $f$ implies
$$\mathrm{ev}_{x_k}([f]_{y_k}) = f(x_k, y_k) \longrightarrow f(x, y) = \mathrm{ev}_x([f]_y),$$
which proves that $\theta(x_k, y_k) \to \theta(x, y)$ in $S_{par}(A)$. 
Since $K$ and $S_{par}(A)$ are compact Hausdorff spaces and $\theta$ is a continuous bijection, $\theta$ is a homeomorphism. 

Finally, when $A$ is equipped with the standard basis $\{e_0=1, e_1=x_1, \dots, e_n=x_n\}$ and $\mathcal{M}_A \cong \mathcal{C}(Y)$ is equipped with the standard coordinate generating tuple $(y_1, \dots, y_m)$, the coordinate map $\iota_y: S_y(A) \to \mathbb{R}^n$ of Remark \ref{rem:SK} satisfies 
$$\iota_y(ev_x) = (ev_x([e_1]_y), \dots, ev_x([e_n]_y)) = (x_1, \dots, x_n) = x,$$
and the base evaluation map $y \mapsto (y_1(y), \dots, y_m(y)) = y$ is $\mathrm{id}_Y$. Thus, the coordinate state space $K_A = \bigcup_{y \in Y}(\iota_y(S_y(A)) \times \{y\})$ is identical to $K$. 

 For any other module basis $\{e_0=1, e_1, \dots, e_n\}$ of $A$ and any other generating tuple $(g_1, \dots, g_m)$ of $\mathcal{M}_A$, the induced map $\Theta: K \to K_A$ is an isomorphism in the sense of Definition \ref{def:part-iso}. 
 Indeed, the generating tuple defines a homeomorphism $\phi: Y \to Y' \subseteq \mathbb{R}^m$ via $\phi(y) = (g_1(y), \dots, g_m(y))$, and for each $y \in Y$, the coordinate map $\iota_y: S_y(A) \to \mathbb{R}^n$ yields:
$$\theta_y(x) := \iota_y(ev_x) = (e_1(x,y), \dots, e_n(x,y)) \in (K_A)_{\phi(y)}$$
Define the global map $\Theta: K \to K_A$ by
$$\Theta(x,y) = (\theta_y(x), \phi(y)).$$
For each fixed $y \in Y$, the map $\theta_y: K_y \to (K_A)_{\phi(y)}$ is affine because each basis function $x \mapsto e_i(x,y)$ is affine on $K_y$. Furthermore, because $\{[1]_y, [e_1]_y, \dots, [e_n]_y\}$ is a basis of $A_y \cong \mathrm{Aff}(K_y)$ and $\mathrm{int}(K_y) \neq \emptyset$, the affine map $\theta_y$ has full rank (its linear part is invertible), making it an affine homeomorphism between the compact convex slices $K_y$ and $(K_A)_{\phi(y)}$. Since each $e_i$ and $g_j$ is continuous, $\Theta$ is a continuous bijection between the compact space $K$ and the Hausdorff space $K_A$, hence a global homeomorphism. Therefore, $\Theta$ is an isomorphism of partially convex sets. 
\end{proof}

\end{document}